\documentclass[11pt,reqno]{amsproc}
\usepackage[margin=1in]{geometry}
\usepackage{amsmath, amsthm, amssymb}
\usepackage[abbrev,lite,nobysame]{amsrefs}
\usepackage[usenames,dvipsnames]{color}
\usepackage{mathtools}
\usepackage{dsfont}
\usepackage{hyperref}

\mathtoolsset{showonlyrefs=true}

\newtheorem{proposition}{Proposition}[section]
\newtheorem{theorem}[proposition]{Theorem}
\newtheorem{corollary}[proposition]{Corollary}
\newtheorem{lemma}[proposition]{Lemma}
\theoremstyle{definition}
\newtheorem{definition}[proposition]{Definition}
\newtheorem{remark}[proposition]{Remark}

\numberwithin{equation}{section}

\newcommand{\loc}{\mathrm{loc}}

\DeclareMathOperator{\tr}{\mathrm{tr}}

\newcommand{\R}{\mathbb{R}}
\newcommand{\T}{\mathbb{T}}

\DeclareMathOperator{\Div}{\mathrm{div}}

\usepackage{bbm}

\newcommand{\ep}{\varepsilon}

\newcommand{\tensor}{\otimes}
\renewcommand{\S}{\mathbb{S}}
\newcommand{\E}{\mathbf{E}}
\renewcommand{\P}{\mathbf{P}}

\newcommand{\MMMsw}[1]{ [\Meas_{t,x,v}]_{*} }

\newcommand{\dee}{\mathrm{d}}

\newcommand{\dt}{\dee t}

\newcommand{\dx}{\dee x}

\newcommand{\Grad}{\nabla}

\newcommand{\norm}[1]{\left\| #1 \right\|}

\newcommand{\Torus}{\mathbb{T}}

\newcommand{\Integers}{\mathbb{Z}}
\newcommand{\Integer}{\mathbb{Z}}

\newcommand{\Natural}{\mathbb{N}}
\renewcommand{\norm}[1]{\|#1\|}
\newcommand{\abs}[1]{\left\vert#1\right\vert}
\newcommand{\set}[1]{\left\{#1\right\}}

\newcommand{\grad}{\nabla}
\newcommand{\diverg}{\textup{div} }

\newcommand{\Real}{\mathbb{R}}
\newcommand{\Ldiv}{L^2_{\mathrm{div}}}
\newcommand{\Hdiv}{H^1_{\mathrm{div}}}
\newcommand{\HH}{\dot{H}}

\def\eps{\varepsilon}
\def\e{{\rm e}}
\def\dd{{\rm d}}

\def\EE{\E}

\newcommand{\brak}[1]{\langle{#1}\rangle}

\begin{document}
\title[The Kolmogorov 4/5 law for 3D Navier-Stokes]{A sufficient condition for the Kolmogorov 4/5 law
for stationary martingale solutions to the 3D Navier-Stokes equations}

\author[{J. Bedrossian, M. Coti Zelati, S. Punshon-Smith, and F. Weber}]{Jacob Bedrossian, Michele Coti Zelati, Samuel Punshon-Smith, and Franziska Weber}

\address{Department of Mathematics, University of Maryland, College Park, MD 20742, USA}
\email{jacob@math.umd.edu}
\email{punshs@math.umd.edu}
\email{frweber@umd.edu}

\address{Department of Mathematics, Imperial College London, London, SW7 2AZ, UK}
\email{m.coti-zelati@imperial.ac.uk}
\thanks{J.B. was supported by NSF CAREER grant DMS-1552826, M.C.Z. was supported by NSF DMS-1713886, and F.W. was supported by the Norwegian Research Council, project number 250302. Additionally, this research was supported in part by NSF RNMS \#1107444 (Ki-Net) \textit{Date:} \today}

\subjclass[2000]{35Q30, 60H30, 76F05}

\keywords{Kolmogorov 4/5 law, turbulence, martingale solutions, Navier-Stokes equations}

\begin{abstract}
 We prove that statistically stationary martingale solutions of the 3D Navier-Stokes equations on $\T^3$ subjected to white-in-time (colored-in-space) forcing  satisfy the Kolmogorov 4/5 law (in an averaged sense and over a suitable inertial range) using only the assumption that the kinetic energy is $o(\nu^{-1})$ as $\nu \rightarrow 0$ (where $\nu$ is the inverse Reynolds number). 
 This  plays the role of a weak anomalous dissipation.  
No energy balance or additional regularity is assumed (aside from that satisfied by all martingale solutions from the energy inequality).
 If the force is statistically homogeneous, then any homogeneous martingale solution satisfies the spherically averaged 4/5 law pointwise in space. 
 An additional hypothesis of approximate isotropy in the inertial range gives the traditional version of the Kolmogorov law. 
 We demonstrate a necessary condition by proving that energy balance and an additional quantitative regularity estimate as $\nu \rightarrow 0$ imply that the 4/5 law  (or any similar scaling law) cannot hold.
\end{abstract}

\maketitle

\section{Introduction}
In this article, we consider the stochastically forced Navier-Stokes equations on $\T^3$
\begin{equation}\label{eq:NSE}\tag{\textsc{NSE}}
\begin{cases}
\partial_t u + u\cdot\nabla u + \nabla p - \nu \Delta u= \partial_t W, \\
 \Div u = 0,
\end{cases}
\end{equation}
describing the motion of a viscous, incompressible fluid on a periodic domain.
The forcing is assumed to be a mean zero, white-in-time and colored-in-space Gaussian process, that can be represented by
\begin{equation}\label{eq:defW} 
\partial_t W(t,x) = \sum_{k=1}^\infty \sigma_k e_k(x) \dd\beta_k(t), 
\end{equation}  
where $\{\beta_k(t)\}$ are independent one-dimensional Brownian motions, $\{e_k\}$ are orthonormal eigenfunctions of the Stokes operator on $\T^3$, and 
$\{\sigma_k\}$ are fixed constants satisfying a coloring condition
\begin{align}
\eps:=\frac{1}{2}\sum_{k=1}^\infty \abs{\sigma_k}^2<\infty.
\end{align}
The parameter $\nu>0$ plays the role of the inverse Reynolds number while the quantity $\ep$ plays the role of the mean energy input by the noise. It is important to note that we do not assume any dependence of $\nu$ on $\ep$, that is, the energy input is {\em independent} of Reynolds number.
We are primarily interested in a special class of solutions to \eqref{eq:NSE}, called (weak) \emph{stationary martingale} solutions, whose existence was proved in \cite{FlandoliGatarek1995},
in connection with scaling laws in turbulence theory. Martingale solutions are the, probabilistically weak, stochastic counterpart of Leray-Hopf solutions
in the deterministic setting, and stationary refers to the property that the paths $u(\cdot+\tau)$ and $u(\cdot)$ coincide in law for each $\tau >0$. 
It is worth noting here that stationary martingale solutions satisfy the energy inequality
\begin{align} 
\E \nu \norm{\grad u}_{L^2_x}^2 \leq \eps. \label{eq:EnergyIneq}
\end{align} 

The theory of turbulence, put forward in the 1930s and 1940s by Taylor, Richardson, Kolmogorov and 
others (see \cite{Frisch1995} and the references therein), revolutionized classical physics and has been 
widely influential in fluid mechanics, atmospheric and ocean sciences, and plasma physics. 
In the sequence of papers \cite{K41a,K41b,K41c}, usually referred to as \emph{K41}, Kolmogorov took essentially 
three basic axioms about the fluid flow and formally derived several predictions about the statistics of solutions. 
One of the axioms he assumed is not supported by experimental studies, which results in deviations from some 
of the predictions of the classical theory due to what is now commonly referred to as \emph{intermittency}; see e.g.  
\cite{K62,AnselmetEtAl1984,FrischParisi1985,StolovitzkyEtAl1992,SreeniKail93,vdWH99,Frisch1995,KanedaEtAl03} and the references therein. In particular, real flows are observed to lack statistical 
self-similarity. Intermittency is also connected to the rare events/large deviations from the average behavior.
Nevertheless, the Kolmogorov \emph{4/5 law}, which we describe below, is independent of the problematic axiom and matches 
very well with experiments. In fact, it is regarded as an `exact' law of turbulence by the physics community.  

Despite its overwhelming influence in physics and engineering, relatively little of turbulence theory has been put on 
mathematically rigorous foundations. The idea that many of the basics of turbulence theory could be refined and 
made mathematically rigorous by studying \eqref{eq:NSE} in the vanishing viscosity limit $\nu \rightarrow 0$ is an old one. 
Probability is the most natural framework: turbulent experiments are not repeatable so one can only hope to build 
theories to predict statistics. 
The statistically stationary viewpoint was taken in K41, and the idealized case of $x \in \Torus^3$ 
seems natural as the starting point. 
Moreover, there exists a far greater range of mathematical tools for understanding the generic behavior of stochastic processes 
rather than deterministic ones, especially in the presence of white-in-time forcing. In the context of fluid dynamics, the 
most striking examples of this are the infinite-dimensional ergodicity and related results; see for example \cite{DD03,HM06,HM08,HM11,KNS18}. 
Normally, studying the incompressible Euler equations, when $\nu = 0$,  is more tractable than the limit $\nu \rightarrow 0$.
However, given recent results on `wild' solutions to the Euler equations \cite{DS09, DS12, DS13, DShprinc17, Isett16,Isett17,BuckmasterEtAl2017}, 
these equations might not have an intrinsic selection principle for solutions obtained as limits of suitable martingale solutions. 
One might expect that an inviscid limit holds, e.g., that statistically stationary solutions to the Navier-Stokes equations converge 
in the vanishing viscosity limit to dissipative statistically stationary solutions of the Euler equations. Nonetheless, 
deducing the required compactness would require us first to understand much better the $\nu > 0$ behavior 
(see \cite{DV17} for discussions on what compactness is required in the deterministic setting and \cite{FGHV16} 
for an example of a shell model in which such results are attainable).

\subsection{The 4/5 law}
Assuming that the flow is isotropic, statistically stationary in space and time, and satisfies a 
suitable \emph{anomalous dissipation} assumption (discussed below), 
Kolmogorov formally made the prediction that over a range of scales $|h| \in [\ell_D, \ell_I]$ with $\ell_D \approx \eps^{-1/4} \nu^{3/4}$, 
the following 4/5 law holds 
\begin{align} 
\E \left(\delta_{h} u \cdot \frac{h}{\abs{h}}\right)^3 \sim -\frac{4}{5} \eps |h|, \label{eq:K41}
\end{align} 
where $\delta_h u(x):= u(x+h)-u(x)$ is the increment by the vector $h\in \R^3$. The quantity appearing on the left-hand side above  is referred to as the third-order longitudinal structure function. 
The range over which viscous effects will be mostly negligible, $[\ell_D, \ell_I]$, is commonly called the \emph{inertial range}. 
The top scale $\ell_I$, called the \emph{integral scale},  is essentially the scale at which the external forcing acts, 
while $\ell_D$, called the \emph{dissipation scale}, is the scale where the viscosity dominates the flow.

In this paper, we aim to provide a mathematically rigorous derivation of the Kolmogorov 4/5 law and the closely related \emph{4/3 law}
for solutions to \eqref{eq:NSE} with minimal hypotheses.  Since the 4/5 law is independent of the precise details of intermittency, 
it is intuitive that we are able to prove it using very few hypotheses, and in particular, we do not require any additional 
regularity other than what follows from the energy inequality \eqref{eq:EnergyIneq}. 

\subsection{Weak anomalous dissipation}
The fundamental axiom of hydrodynamic turbulence is the assumption of \emph{anomalous dissipation}, which is meant to 
describe the tendency of the nonlinearity to accelerate the dissipation of energy so much that the Navier-Stokes equations 
\eqref{eq:NSE} can balance the external force independently of the Reynolds number.  
The non-vanishing of energy dissipation $\nu \EE \norm{\grad u}_2^2$ in the limit $\nu \rightarrow 0$ is often used as 
a definition, however this is not the correct definition for stationary solutions to \eqref{eq:NSE}. 
To illustrate this, consider stationary solutions to the heat equation 
\begin{align} 
\partial_t v - \nu \Delta v = \partial_t W, \label{eq:heat}
\end{align} 
for which  It\^o's formula provides the energy balance
\begin{align*} 
\nu \EE \norm{\grad v}_{L^2_x}^2 = \eps.  
\end{align*} 
The energy dissipation $\nu \EE \norm{\grad v}^2_{L^2_x}$ does not vanish as $\nu\to 0$, but there is nothing anomalous 
or nonlinear occurring here. 
A priori, one only has an energy inequality \eqref{eq:EnergyIneq} for martingale solutions to \eqref{eq:NSE} and it is an interesting question to wonder whether equality holds or not (see \cite{LS18} and references therein for sufficient conditions in the deterministic case). 
Equality or inequality in \eqref{eq:EnergyIneq} is a question connected to the types of potential finite-time singularities for $\nu > 0$ and hence is a property related mostly to rare/intermittent events.  
As discussed above, this suggests that the 4/5 law in the form \eqref{eq:K41} should hold independently of whether or not equality is achieved in \eqref{eq:EnergyIneq}\footnote{There is, however, the classical inviscid limit version of the anomalous dissipation: that $\set{u}_{\nu > 0}$ converges to statistical solutions of Euler with non-vanishing energy flux which balances $\eps$. A priori, this property is far stronger than Definition \ref{def:WAD}.}. 

Let us suggest a more useful alternative (stronger variants have been taken in the derivation of the 4/5 law in for 
example \cite{Frisch1995,Nie1999}). Thanks to the exact form of the invariant measure, for the heat equation \eqref{eq:heat}, we have 
\begin{align}
\nu \EE \norm{v}_{L^2_x}^2 \approx \eps. 
\end{align} 
For the Navier-Stokes equations, we have the similar upper bound $\nu \EE \norm{u}_{L^2_x}^2 \lesssim \eps$ from Poincar\'e's inequality and \eqref{eq:EnergyIneq}. 
However, we expect that kinetic energy is driven to smaller scales as $\nu \rightarrow 0$, drastically increasing the amount of energy that is dissipated. 
Hence, for statistically stationary solutions to \eqref{eq:NSE}, it is natural to take following definition of anomalous dissipation.
\begin{definition}[Weak anomalous dissipation] \label{def:WAD}
We say that a sequence $\set{u}_{\nu > 0}$ of stationary martingale solutions 
satisfies \emph{weak anomalous dissipation} if
\begin{align}\label{eq:WAD}\tag{\textsc{WAD}}
\lim_{\nu \rightarrow 0} \nu \EE \norm{u}_{L^2_x}^2 = 0.
\end{align}
\end{definition}
In the physics literature (see e.g. \cite{Frisch1995}), it is usually implicitly assumed that $\EE \norm{u}_{L^2_x}^2$ is bounded independently 
of the Reynolds number, which would more closely match the idea that the energy input should be totally balanced uniformly in the Reynolds number. 
In \cite{Nie1999}, the assumption $\nu^{1/2} \EE \norm{u}_{L^2_x}^2 \to 0$ as $\nu\to 0$ is used, which appears to be 
natural for deterministic forcing. However, for solutions to \eqref{eq:NSE}, these stronger assertions are unnecessary to deduce the 4/5 law, one only needs that the 3D Navier-Stokes equations dissipate significantly more energy than the heat equation.  The main result of this paper is to show that only \eqref{eq:WAD} is required to deduce the 4/5 law in its averaged form.

\begin{remark}
Consider a passive scalar advected by an incompressible, time and $\nu$-independent smooth velocity $v$ such as 
\begin{align}
\partial_t f+v\cdot \nabla f-\nu \Delta f=\partial_t W.
\end{align}
If $v$ is weakly mixing (or, more generally, relaxation enhancing \cite{CKRZ}), then it was proved in \cite{BCZGH} that \eqref{eq:WAD}
holds for $f$, namely
\begin{align}
\lim_{\nu \rightarrow 0} \nu \EE \norm{f}_{L^2_x}^2 = 0.
\end{align}
\end{remark}

\subsection{Main results}
Denoting by
\begin{align}
\delta_h u(x)= u(x+h)-u(x), \qquad h\in \R^3,
\end{align}
the increment of $u$ with respect to the vector $h$,
we define for any $\ell\in \R$ the averaged structure functions
\begin{subequations}\label{eq:structure}
\begin{align}
 S_0(\ell) & = \frac{1}{4\pi}\E\int_{\S^2} \int_{\T^3} |\delta_{\ell \hat{n}}u|^2 \delta_{\ell \hat{n}} u\cdot \hat{n} \,\dx \dee S(\hat{n}), \label{eq:S_0}\\ 
 S_{||}(\ell) & = \frac{1}{4\pi}\E\int_{\S^2} \int_{\T^3}  (\delta_{\ell \hat{n} } u \cdot \hat{n})^3\dx \dee S(\hat{n})\label{eq:Sp}. 
\end{align} 
\end{subequations}
Our main result reads as follows.
\begin{theorem} [Averaged 4/3 and 4/5 laws] \label{thm:T}
 Let $\set{u}_{\nu > 0}$ be a sequence of stationary martingale  solutions to \eqref{eq:NSE},
 which satisfy \eqref{eq:WAD}, and let  $S_0(\ell)$  and  $S_{||}(\ell)$ be the corresponding structure functions 
 as defined in \eqref{eq:structure}. Then the following holds true. 
 \begin{itemize}
  \item (4/3 law) There exists $\ell_D = \ell_D(\nu)$ with $\displaystyle\lim_{\nu \rightarrow 0} \ell_D = 0$ such that
        \begin{align}
         \lim_{\ell_I \rightarrow 0} \limsup_{\nu \rightarrow 0}  \sup_{\ell \in [\ell_D, \ell_I]}  \abs{\frac{S_{0}(\ell)}{\ell} + \frac{4}{3} \eps} = 0. \label{43Law}
        \end{align}
    \item (4/5 law) There exists $\ell_D = \ell_D(\nu)$ with $\displaystyle\lim_{\nu \rightarrow 0} \ell_D = 0$ such that
        \begin{align}
         \lim_{\ell_I \rightarrow 0} \limsup_{\nu \rightarrow 0}  \sup_{\ell \in [\ell_D, \ell_I]}  \abs{\frac{S_{||}(\ell)}{\ell} + \frac{4}{5} \eps} = 0. \label{45Law}
        \end{align}
\end{itemize}
\end{theorem}
To the best of our knowledge, Theorem \ref{thm:T} proves the 4/5 and the 4/3 laws under significantly weaker assumptions than those
previously appearing in the literature. 
There is a number of related works on the 4/5 law in the deterministic setting.
The closest analogue to ours is \cite{Nie1999}, which considers deterministic, smooth solutions with stronger anomalous dissipation hypotheses.
Another related work is that by Duchon and Robert~\cite{Duchon2000}, where an expression for the \emph{defect measure} 
(measuring the defect from energy conservation) is derived for solutions of the deterministic Navier-Stokes equations and 
$L_{t,x}^3$-solutions of the incompressible Euler equations which arise as strong limits of sequences of Leray-Hopf solutions 
of the Navier-Stokes equations.
Under a strong $L_{t,x}^3$ inviscid limit assumption, Eyink \cite{Eyink2003} establishes expressions for the defect measure for weak solutions of the deterministic  incompressible Euler equations that relate to local deterministic versions of the structure functions~\eqref{eq:structure} above. 
In a similar framework, the recent paper \cite{Drivas18} derives a Lagrangian analogue of the 4/5 law.
Lastly, we mention that under much stronger assumptions, approximate scaling laws for the \emph{second} order structure function 
(e.g. when the $(\delta_{\ell \hat{n} } u \cdot \hat{n})^3$ is replaced by $(\delta_{\ell \hat{n} } u \cdot \hat{n})^2$ in \eqref{eq:Sp}) have been 
derived in \cite{CSint14} for deterministic forces and in \cite{FlandoliEtAl08} in the stochastic case. In either case, the second order structure
function requires intermittency corrections.

Before stating further results, some remarks are in order.
\begin{remark} 
The proof chooses $\ell_D$ such that 
\begin{align}
\lim_{\nu \rightarrow 0} \frac{\nu \EE \norm{u}_{L^2_x}^2}{\ell_D^2} = 0.
\end{align} 
Note that $\ell_D$ should be interpreted only as an upper bound on the dissipative length scale. 
\end{remark} 

\begin{remark}
    The idea of averaging over the sphere to deal with potential anisotropy was also discussed in~\cite{Nie1999,Arad1999,KurienSreenivasan2000}, and in~\cite{TaylorKurienEyink2003} angle-averaging is used to recover the 4/5 law in numerical simulations.     
\end{remark} 

\begin{remark} 
It is an open problem to prove  existence of solutions satisfying \eqref{eq:WAD}. However,  it should be noted that this property cannot hold 
for all solutions. Indeed, consider the case that $\partial_t W = e_j \dee\beta(t)$, for a single mode. In this case, it is easy to check 
that the solution to the stochastically forced heat equation \eqref{eq:heat} is also a stationary solution of \eqref{eq:NSE}. In fact, the 
solution is $u(t,x) = Z(t) e_j(x)$, where $Z$ is the Ornstein-Uhlenbeck process $\dee Z = -\nu \abs{j}^2 Z  \dee t  + \dee\beta(t)$, which 
can be seen not to satisfy \eqref{eq:WAD}. 
It is natural to conjecture that \eqref{eq:WAD} holds whenever the stochastic forcing satisfies some suitable H{\"o}rmander bracket condition, 
in a similar way as it implies ergodicity in 2D Navier-Stokes \cite{HM11}. 
\end{remark} 

The proof of Theorem \ref{thm:T} also yields the following  a priori estimate uniformly in $\nu$ without any assumptions on the solution, however, it is not clear what any potential implications are. 
Nevertheless, $\nu$-independent estimates on statistically stationary solutions are few and far between, so we have chosen to include this observation.  

\begin{proposition}  \label{prop:Triv} 
Let $\set{u}_{\nu > 0}$ be a sequence of stationary martingale  solutions to \eqref{eq:NSE},
and let  $S_0(\ell)$  and  $S_{||}(\ell)$ be the corresponding structure functions.
Then, for every fixed $\nu>0$, there holds
        \begin{align}
         -\frac{8}{3}\eps \leq \liminf_{\ell \rightarrow 0}\frac{S_0(\ell)}{\ell} \leq \limsup_{\ell \rightarrow 0} \frac{S_0(\ell)}{\ell} \leq 0. 
        \end{align}        
A similar statement holds also for $S_{||}$.
\end{proposition}

Note that if statistically stationary solutions were smooth, we would have 
\begin{align}
\lim_{\ell \rightarrow 0} \frac{S_0(\ell)}{\ell} = 0.
\end{align} 

\subsection{Statistical symmetries}
A second axiom often assumed when studying turbulence is that the symmetries of translation and rotation are 
recovered, in a statistical sense, in the inertial range.
Such recovery of symmetry has never been proved rigorously in any context to our knowledge. 
However, if the force is statistically homogeneous (a process $f$ defined on $\Torus^3$ is statistically homogeneous if the paths $f(\cdot)$ and $f(\cdot+y)$ coincide in law for any $y \in \Torus^3$) then one can check that there exists homogeneous and stationary martingale solutions (see Remark \ref{rem:homogeneous} below).
In this case, we can dispense with the space average in Theorem \ref{thm:T}. 
\begin{theorem} [Homogeneous 4/3 and 4/5 laws] \label{thm:HomT}
 Let $\set{u}_{\nu > 0}$ be a sequence of homogeneous and stationary martingale  solutions to \eqref{eq:NSE},
 which satisfy \eqref{eq:WAD}. Then the following holds true:
 \begin{itemize}
  \item (4/3 law) There exists $\ell_D = \ell_D(\nu)$ with $\displaystyle\lim_{\nu \rightarrow 0} \ell_D = 0$ such that
        \begin{align}
         \lim_{\ell_I \rightarrow 0} \limsup_{\nu \rightarrow 0}  \sup_{\ell \in [\ell_D, \ell_I]}  \abs{ \EE \frac{1}{4\pi \ell} \int_{\S^2} \abs{\delta_{\ell \hat{n}} u}^2 \delta_{\ell \hat{n}} u \cdot \hat{n} \dee S(\hat{n})  + \frac{4}{3} \eps} = 0.
        \end{align}
    \item (4/5 law) There exists $\ell_D = \ell_D(\nu)$ with $\displaystyle\lim_{\nu \rightarrow 0} \ell_D = 0$ such that
        \begin{align}
         \lim_{\ell_I \rightarrow 0} \limsup_{\nu \rightarrow 0}  \sup_{\ell \in [\ell_D, \ell_I]}  \abs{ \EE \frac{1}{4\pi \ell} \int_{\S^2} \abs{\delta_{\ell \hat{n}} u}^2 \delta_{\ell \hat{n}} u \cdot \hat{n} \dee  S(\hat{n}) + \frac{4}{5} \eps} = 0.
        \end{align}
\end{itemize}
\end{theorem}
Proposition \ref{prop:Triv} also holds in the homogeneous case, in which the structure function \eqref{eq:S_0} need not be integrated in $x$.

It is natural to ask what kind of approximate symmetries imply the 4/5 law expected in homogeneous, isotropic turbulence. 
The following definition assumes spatial homogeneity already, but generalizations are natural. 

\begin{definition}[Approximate (inertial range) isotropy] \label{def:ApproxSym}
 We say that a sequence $\set{u}_{\nu > 0}$ of homogeneous and stationary martingale solutions  to \eqref{eq:NSE}
 is \emph{approximately isotropic} if there exists $\ell_D = \ell_D(\nu)$ with $\displaystyle\lim_{\nu \rightarrow 0} \ell_D(\nu) = 0$ 
 such that the following holds: for every continuous function $\phi:\Real^3 \times \Real^3 \rightarrow \Real$ such that 
 \begin{align}
 \abs{ \EE\frac{1}{4\pi}\int_{\S^2} \phi(\delta_{\ell \hat{n}} u, \ell \hat{n}) \dee S(\hat{n}) } \lesssim \gamma(\ell)
 \end{align}
 for some strictly monotone increasing, continuous $\gamma:\Real_+ \rightarrow \Real_+$ with $\gamma(0) = 0$, we have  
 \begin{align}
\lim_{\ell_I \rightarrow 0}\limsup_{\nu \rightarrow 0} \sup_{\ell \in [\ell_D,\ell_I]} \frac{1}{\gamma(\ell)}\frac{1}{4\pi}\int_{\S^2} \abs{ \EE \phi(\delta_{\ell \hat{n}} u,\ell \hat{n}) - \EE\frac{1}{4\pi}\int_{\S^2} \phi(\delta_{\ell \hat{n}'} u,\ell \hat{n}') \dee S(\hat{n}')  } \dee S(\hat{n}) = 0. 
\end{align} 
\end{definition}
This definition together with Theorem \ref{thm:T} trivially implies the following variant of the 4/5 law, which is the classical statement often seen in physics texts. 

\begin{corollary}[4/5 law for homogeneous, approximately isotropic turbulence] \label{cor:ptwise45}
 Let $\set{u}_{\nu > 0}$ be a sequence of homogeneous, approximately isotropic, and  stationary martingale solutions to \eqref{eq:NSE},
 which satisfy \eqref{eq:WAD}.
  Then, there exists $\ell_D = \ell_D(\nu)$ with $\displaystyle\lim_{\nu \rightarrow 0} \ell_D(\nu) = 0$ such that
        \begin{align}
         \lim_{\ell_I \rightarrow 0} \limsup_{\nu \rightarrow 0}  \sup_{\ell \in [\ell_D, \ell_I]} \int_{\S^2} \abs{ \EE \frac{1}{\ell} (\delta_{\hat{n} \ell} u\cdot \hat{n})^3 + \frac{4}{5} \eps} \dee S(\hat{n}) = 0.
        \end{align}
\end{corollary}

\subsection{Necessary conditions}
Finally, we state conditions that are necessary for any scaling laws resembling the 4/5 and 4/3 laws to hold.
In \cite{DV17} necessary conditions are proved for deterministic solutions obtained in the inviscid limit in $\Torus^3$ whereas \cite{CV18} discusses similar necessary conditions for deterministic solutions obtained in the inviscid limit on bounded domains (see also the references therein for further discussions).  
 
\begin{theorem}\label{thm:Nec}
Suppose $\set{u}_{\nu > 0}$ is a sequence of stationary martingale solutions to \eqref{eq:NSE} that satisfies the following conditions.
\begin{itemize} 
\item Energy balance: for every $\nu$ sufficiently small, there holds
\begin{align}\label{eq:condition1}
\nu \EE \norm{\grad u}_{L^2_x}^2 = \eps.
\end{align} 
\item Regularity: there exist $ C > 0$ and $s > 1$ such that 
\begin{align}\label{eq:condition2}
\sup_{\nu\in (0,1)}\nu \EE \norm{\abs{\grad}^{s} u}_{L^2_x}^2 \leq C.
\end{align}
\end{itemize} 
Then, there holds
\begin{align}\label{eq:scalinglaw}
\lim_{\ell \to 0}\sup_{\nu\in (0,1)}\left(\abs{\frac{S_{||}(\ell)}{\ell}} +\abs{\frac{S_0(\ell)}{\ell}}\right)=  0.
\end{align} 
\end{theorem}

\begin{remark}
It is easy to check from the proof that the regularity condition \eqref{eq:condition2} can be weakened to
\begin{align}
\lim_{\ell\to 0}\sup_{|h|\leq \ell }\sup_{\nu\in (0,1)}\E \nu \|\nabla \delta_h u \|_{L^2_x}^2= 0.
\end{align}
This is essentially $L^2_{\omega,t}$ precompactness in space of $\sqrt\nu\nabla u$.
\end{remark}

If we assume that \eqref{eq:condition2}  holds for $s>5/4$, then \eqref{eq:condition1} is superfluous. Indeed, thanks to the embedding
$L^2_t H^{s}_x \cap L^\infty_t L^2_x \hookrightarrow  L^3_t W^{\sigma,3}_x$ for some $\sigma > 1/3$, it is possible 
to adapt the arguments in \cite{CET,Eyink94} to the case of \eqref{eq:NSE}, and obtain conservation of energy. 
Hence we have the following corollary. 

\begin{corollary}
Let $\set{u}_{\nu > 0}$ be a sequence of stationary martingale solutions to \eqref{eq:NSE} such that for some 
$c \neq 0$ and some $\ell_D(\nu)$ with $\displaystyle\lim_{\nu \rightarrow 0} \ell_D = 0$ there holds
\begin{align}
\lim_{\ell_I \rightarrow 0}\lim_{\nu \rightarrow 0} \sup_{\ell \in [\ell_D,\ell_I]} \abs{\frac{S_0(\ell)}{\ell} +  c} = 0.
\end{align}
Then, for all $s > 5/4$, 
\begin{align}
\liminf_{\nu \rightarrow 0} \EE \nu \norm{\abs{\grad}^{s}u}_{L_x^2}^2 = \infty.
\end{align} 
\end{corollary}

\begin{remark}
If $\inf_{\nu}\EE \nu \norm{\grad u}^2_{L_x^2} >0$, then \eqref{eq:WAD} implies that 
\begin{align}
\liminf_{\nu \rightarrow 0} \EE \nu \norm{\abs{\grad}^{s}u}_{L_x^2}^2 = \infty, \qquad \forall s>1.
\end{align} 
Indeed, by Sobolev interpolation we have for $\theta=1-1/s$
\begin{align}
\EE \norm{\grad u}_{L^2_x}^2 \leq \EE \norm{u}_{L^2_x}^{2\theta}\norm{\abs{\grad}^s u}_{L^2_x}^{2(1-\theta)} \leq \left(\EE \norm{u}_{L^2_x}^{2}\right)^\theta  \left(\EE \norm{\abs{\grad}^s u}_{L^2_x}^2\right)^{1-\theta}.
\end{align}
Hence, under the assumption of energy balance \eqref{eq:condition1},  \eqref{eq:WAD}  is stronger than \eqref{eq:condition2}.
Nevertheless, we conjecture that \eqref{eq:WAD} is both necessary and sufficient for the 4/5 law, at least if energy balance holds.
\end{remark}

\subsection{Notations and conventions}
We write $f \lesssim g$ if there exists $C > 0$ such that $f \leq C g$ (and analogously $f \gtrsim g$). We write $f \approx g$ if there exists $C > 0$ such that $C^{-1}f \leq g \leq Cf$.
Furthermore, we use hats to denote vectors with unit length, that is, if $h\neq 0$ is some vector, then $\hat{h}=h/|h|$.
We will also sometimes abbreviate $L^p$, $W^{s,p}$ and $C^k$-spaces in the variables $(\omega,t,x,\ell)\in \Omega\times[0,T]\times\T^3\times\R_+$ in the following way:
\begin{equation}
L^p_x:=L^p(\T^3),\quad W^{s,p}_x:= W^{s,p}(\T^3),\quad L^q_t:=L^q([0,T]),\quad
 L^r_\omega:=L^r(\Omega),\quad C^k_\ell := C^k(\R_+),
\end{equation}
etc., or use combinations of these, for example, $L^q_tL^p_x$ denotes the space $L^q([0,T];L^p(\T^3))$ and $L^p_{t,x}$ denotes $L^p([0,T]\times\T^3)$.
We denote by  $\Ldiv$ and $\Hdiv$ the completions of divergence-free smooth functions with zero average (denoted by $ C_\diverg^{\infty}$) with respect to the $L^2$ and $H^1$-norms. For any $\alpha \in \Real$, $\HH_x^\alpha$ will denote the usual homogeneous Sobolev spaces. With a slight abuse of notation we will say a vector field $u(x)\in \R^3$ or a rank two tensor field $A\in \R^{3\times 3}$ belongs to a space $X$ if each component $u^i$ and $A_{ij}$ belongs to $X$. 

We will also make frequent use of component-free tensor notation. Specifically, given any two vectors $u$ and $v$ we will denote $u\tensor v$ the rank two tensor with components $(u\tensor v)_{ij} = u^iv^j$. Moreover given any two rank two tensors $A$ and $B$ we will denote $:$ the Frobenius product defined by, $A:B = \sum_{i,j} A_{ij}B_{ij}$ and the norm $|A| = \sqrt{A:A}$.

\section{Preliminaries}\label{sec:Prelim}

\subsection{Stationary martingale solutions}

As mentioned, martingale solutions to \eqref{eq:NSE} are probabilistically weak analogues to the Leray-Hopf weak solutions to the deterministic Navier-Stokes equations. By probabilistically weak we mean that instead of viewing the noise
\[
	\partial_t W(t,x) = \sum_k \sigma_k e_k(x) \dee\beta_k(t),
\]
and its associated filtered probability space, as an input to the problem, it is actually an output, to be solved for along with the the velocity field $u$. Indeed, a martingale solution consists of a stochastic basis $(\Omega, (\mathcal{F}_t)_{t\in[0,T]},\P, \{\beta_k\})$ along with a process $u:\Omega\times[0,T]\to \Ldiv$ that satisfies the Navier-Stokes equation in the sense of distributions. It is important to remark that, in general the filtration $(\mathcal{F}_t)_{t\in[0,T]}$ may {\it not} be the canonical filtration generated by the Brownian motions $\{\beta_k(t)\}$. Moreover any two martingale solutions may have different stochastic bases. Martingale solutions were first constructed in \cite{BT73}, and stationary martingale solutions were later constructed in \cite{FlandoliGatarek1995}.

More precisely a martingale solution to \eqref{eq:NSE} is defined as follows:
\begin{definition}[Martingale Solution]
\label{def:weakmartingalesol} A \emph{martingale solution} to \eqref{eq:NSE} on $[0,T]$ consists of a stochastic basis
$(\Omega, (\mathcal{F}_t)_{t\in [0,T]},\P, \{\beta_k\})$ with a complete right-continuous filtration and an $\mathcal{F}_t$-progressively measurable stochastic process 
\[
u:[0,T]\times \Omega\to \Ldiv,
\]
 such that:
\begin{itemize}
	\item  $u$ has sample paths in
	\begin{equation}
	C_t(\HH_x^\alpha)\cap L^\infty_t\Ldiv\cap L_t^2\Hdiv,\text{ for some } \alpha < 0.
	\end{equation}
	\item For all $t\in[0,T]$ and all $\varphi \in C_\diverg^{\infty}$, the following identity holds $\P$ almost surely, 
	\begin{equation}\label{eq:weak-form-NS}
	\begin{split}
	&\int_{\T^3} u(t)\cdot\varphi \,\dd x +\nu\int_0^t\int_{\T^3}\Grad u(s):\Grad \varphi \,\dd x \dd s+\int_0^t\int_{\T^3}\left[(u(s)\cdot \Grad) u(s)\right] \cdot \varphi \,\dd x \dd s\\
	&\qquad\qquad=\int_{\T^3}u_0 \cdot\varphi \,\dd x + \sum_k \int_0^t \int_{\T^3} \sigma_k e_k\cdot\varphi\, \dd x \dd \beta_k(s).
	\end{split}
	\end{equation}
\end{itemize}
If there is no risk of confusion, we will simply say that $u$ is a martingale solution without making reference to the stochastic basis.
\end{definition}

We also have the corresponding definition for stationary processes:

\begin{definition}[Stationary Martingale Solution]
\label{def:statweakmartingalesol} A \emph{stationary martingale solution} to \eqref{eq:NSE} $[0,\infty)$ consists of a stochastic basis
$(\Omega, (\mathcal{F}_t)_{t\geq 0},\P, \{\beta_k\})$
with a complete right-continuous filtration and an $\mathcal{F}_t$-progressively measurable stochastic process 
\[
u:[0,\infty)\times \Omega\to \Ldiv,
\] such that:
\begin{itemize}
	\item  For each $T\geq 0$, $u|_{[0,T]}$ has sample paths in
	\begin{equation}
	C_t(\HH_x^\alpha)\cap L^\infty_t\Ldiv\cap L_t^2\Hdiv,\text{ for some } \alpha < 0.
	\end{equation}
	\item the law of the path $u(\cdot+ \tau)$ is that same as the law of $u(\cdot)$ for each $\tau \geq 0.$
	\item For all $t\geq 0$ and all $\varphi \in C_\diverg^{\infty}$, identity \eqref{eq:weak-form-NS} holds $\P$ almost surely.
\end{itemize}
If there is no risk of confusion, we will simply say that $u$ is a stationary martingale solution without making reference to the stochastic basis.
\end{definition}

If, in addition to either definition above, the velocity field $u(t,\cdot+h)$ is equal in law to $u(t,\cdot)$ for every $h\in \R^3$, then we say that 
$u$ is a \emph{stationary and homogeneous martingale solution}.
\medskip

The following existence theorem follows easily from~\cite{FlandoliGatarek1995}.

\begin{theorem}\label{thm:flan-gat} For each $u_0 \in \Ldiv$ and $T\geq 0$ there exists a martingale solution $u$ on $[0,T]$ to \eqref{eq:NSE} satisfying the energy inequality for $t_2 > t_1$, 
\begin{align} 
\frac{1}{2}\EE\norm{u(t_2)}_{L^2_x}^2 - \frac{1}{2}\EE\norm{u(t_1)}_{L^2_x}^2 + \nu\EE \int_{t_1}^{t_2} \norm{\grad u(s)}_{L^2_x}^2 \dee s \leq \eps(t_2-t_1). 
\end{align}
Additionally, there exists a exists a stationary martingale solution $u$ on $[0,\infty)$ to \eqref{eq:NSE} satisfying the stationary energy inequality
\[
	\nu\E \|\nabla u(s)\|^2_{L^2_x} \leq \ep.
\]
Furthermore, if the force is homogeneous, then there exists stationary homogeneous martingale solutions.
\end{theorem}
 \begin{remark}\label{rem:homogeneous}
	One can write the force \eqref{eq:defW} in the form 
	\begin{align*}
	W(t,x) = \sum_{k \in \Integers^3} \alpha_k \cos (k\cdot x) \beta_k^1(t) + \gamma_k \sin (k\cdot x) \beta_k^2(t), 
	\end{align*}
	for independent Brownian motions $\beta_k^{i}$ and for vectors $\alpha_k,\gamma_k$ satisfying $\alpha_k \cdot k = \gamma_k \cdot k = 0$ and $\abs{\alpha_k}^2 + \abs{\gamma_k}^2 = \sigma_k^2$. The force is homogeneous if $\abs{\alpha_k}^2 = \abs{\gamma_k}^2$ for all $k \in \Integer^3$. Existence of stationary, homogeneous martingale solutions can be proved two ways: (A) by a Galerkin approximation, using that each truncation will result in a stationary, homogeneous finite dimensional approximation or (B) averaging path measures over translations as in \cite{VF}. Note that both cases require a Skorokhod embedding argument.  
 \end{remark}
A priori estimates can also be deduced for these martingale solutions. Indeed, Theorem 3.1 in~\cite{FlandoliGatarek1995} implies that martingale solutions satisfy the following a priori estimates, 
\begin{align}\label{eq:apriori-bounds}
\E\|u\|_{L^\infty_t L^2_x}^r <\infty, \quad \E\|\nabla u\|_{L^2_{t,x}}^2 <\infty, \quad \forall  r\in [1,\infty),
\end{align}
where the $L^q_t$ norms are taken over any time interval $[0,T]$ for $T < \infty$.
These two estimates combined provide the following a priori estimate. 
\begin{proposition}\label{prop:apriori}
 Any martingale solution $u$ satisfies
 \begin{align}
  \E\|u\|_{L^q_t L^p_x}^r < \infty
 \end{align}
 for each $p,q \in [1,\infty]$, $r\in[1,\infty)$ satisfying
 \begin{align}
  \frac{2}{q} + \frac{3}{p} \geq \frac{3}{2},\quad\text{and}\quad \frac{2}{r} + \frac{3}{p} > \frac{3}{2}.
 \end{align}
\end{proposition}
\begin{proof}
 We begin by interpolation (see e.g. \cite[Theorem 5.8]{Adams}), for $\P\tensor \dt$ almost every $(\omega,t)\in \Omega\times[0,T]$ we have
 \begin{align}
  \|u(t,\omega)\|_{L^p_x} \lesssim \|u(t,\omega)\|_{L^2_x}^{\theta}\|\nabla u(t,\omega)\|^{1-\theta}_{L^2_x},
 \end{align}
and $\theta \in [0,1]$ satisfies
 \begin{align}
  \theta = \frac{3}{p} - \frac{1}{2}.
 \end{align}
 Next using H\"older's inequality in time gives
 \begin{align}
  \|u(\omega)\|_{L^q_t L^p_x} \lesssim \|u(\omega)\|^\theta_{L^\infty_t L^2_x}\|\nabla u(\omega)\|_{L^2_{t,x}}^{1-\theta}
 \end{align}
 where
 \begin{align}
  \frac{1}{q} \geq \frac{1-\theta}{2} = \frac{3}{4} - \frac{3}{2p}\quad \Rightarrow\quad \frac{2}{q} + \frac{3}{p} \geq \frac{3}{2}.
 \end{align}
 Finally, applying H\"older's inequality in expectation, gives
 \begin{align}
  (\E\|u\|_{L^q_t L^p_x}^r)^{\frac{1}{r}} \lesssim (\E\|u\|_{L^\infty_t L^2_x}^s)^{\frac{1}{s}}(\E\|\nabla u\|_{L^2_{t,x}}^2)^{\frac{1}{2}}
 \end{align}
 where
 \begin{align}
  \frac{1}{r} > \frac{1-\theta}{2} = \frac{3}{4} -\frac{3}{2p}\quad \Rightarrow\quad \frac{2}{r} + \frac{3}{p} > \frac{3}{2},
 \end{align}
 and $s\geq 1$ is given by
 \begin{align}
  \frac{1}{s} = \frac{1}{r}- \frac{1-\theta}{2} > 0.
 \end{align}
\end{proof}

Note that if $u$ is statistically stationary in time, then for every $T>0$ there holds the relation
\begin{equation}\label{eq:timestat} 
\E \norm{u}_X^q = \E \frac{1}{T}\int_0^T \norm{u(t)}_X^{q} \dd t, 
\end{equation} 
which means there is a one-to-one correspondence between $q$-th moments of $X$ and $L_t^{q} X$ regularity. Hence, invoking Proposition \ref{prop:apriori}, stationary martingale solutions satisfy 
\begin{equation} 
\E \norm{u}_{L_x^p}^{q} < \infty, \quad \text{for}\quad \frac{2}{q} + \frac{3}{p} > \frac{3}{2}.
\end{equation} 

\begin{remark}[Martingale and stationary suitable weak solutions]
Existence of (stationary) martingale  \emph{suitable} solutions has been proved in \cite{Romito2010}. 
These are solutions which satisfy an additional local energy inequality that can be used to prove partial regularity results 
in the spirit of Caffarelli-Kohn-Nirenberg~\cite{CaffarelliKohnNirenberg1982} for the stochastic 
Navier-Stokes equations, see~\cite{FlandoliRomito2002}. The versions of the 4/3 and 4/5 law that we 
prove in the following sections hold for any type of martingale solutions, and the local energy inequality
is not used in the proofs.
\end{remark}

\subsection{Regularity and a priori estimates for structure functions}
We now use the previous regularity estimates to deduce useful properties on the structure functions in \eqref{eq:structure}.
\begin{lemma}\label{lem:S0cont}
Let $\set{u}_{\nu > 0}$ be a sequence of (not necessarily stationary) martingale solutions to \eqref{eq:NSE} on $[0,T]$. 	
Then for each $\nu > 0$, the functions
\begin{align}
\ell\mapsto\frac{1}{T}\int_0^T S_0(t,\ell)\dee t, \qquad  \ell\mapsto\frac{1}{T}\int_0^T S_{||}(t,\ell)\dee t,
\end{align}
are continuous on $(0,\infty)$ and  bounded via 
\begin{equation}
\sup_{\ell \in (0,1)}\abs{ \frac{1}{T}\int_0^T S_0(t,\ell)\dee t } + \sup_{\ell \in (0,1)} \abs{ \frac{1}{T}\int_0^T S_{||}(t,\ell)\dee t } \lesssim \E\norm{u}_{L^3_{t,x}}^3. \label{ineq:Sbds} 
\end{equation}
Moreover, if $u$ is stationary, then $S_0$ and $S_{||}$ are time-independent, and for each $\nu > 0$ fixed, both are continuous and 
 \begin{align}
 \lim_{\ell\to 0} S_0(\ell) = \lim_{\ell \to 0} S_{||}(\ell) =0.
 \end{align}
\end{lemma}
\begin{proof}
The proof is the same for both $S_0$ and $S_{||}$, so we just show it for $S_0$.
Proposition \ref{prop:apriori} implies that
\begin{equation}
\E\int_0^T\int_{\T^3}|\delta_{\ell \hat{n}} u|^3 \dee x \dee t \lesssim  \E\norm{u}_{L^3_{t,x}}^3 < \infty.
\end{equation}
Hence, $\delta_{\ell \hat{n}} u \in L_{\omega,t,x}^3$ uniformly in $\ell$ and $\hat{n}$, and by Fubini's theorem we may swap integrals and deduce
that $S_0(t,\ell)$ and $S_{||}(t,\ell)$ are in $L^1_t$. From this, \eqref{ineq:Sbds} follows.
The continuity holds because $u \in L^{3}_{\omega,t,x}$ and shifts are continuous in $L_x^p$, i.e.,  for any $h\in \R^3$,
\begin{equation}
\lim_{\ell\to 0}\E\int_{0}^T\int_{\T^3}|u(t,x+\ell h)-u(t,x)|^p \dee x \dee t = 0.
\end{equation}
For time-stationary functions, the same holds without integrating in time by \eqref{eq:timestat}. 
\end{proof}

Another important quantity related to statistical averages is the two point correlation matrix for the velocity fields
\begin{align}\label{eq:GAMMA}
\Gamma(t,h) := \E \int_{\T^3} u(t,x)\tensor u(t,x+h)\,\dx,
\end{align}
which, written component-wise, reads
	\begin{equation}
	\Gamma_{ij}(t,h)=\E\int_{\T^3} u^{i}(t,x) u^{j}(t,x+h) \dee x,\quad i,j=1,2,3.
	\end{equation}
The following is a regularity result that is useful later. 
\begin{lemma}\label{lem:gammaref}
	For every $i,j=1,2,3$, the two point correlation functions $\Gamma_{ij}$
	are uniformly bounded and continuous. Moreover the time averages
	\begin{align}
	h\mapsto\frac1T\int_0^T\Gamma_{ij}(t,h)\dee t
	\end{align}
are $C^2(\R^3)$ for each $\nu > 0$. If $u$ is time-stationary, then $\Gamma_{ij}(t,h)\equiv \Gamma_{ij}(h)\in C^2(\R^3)$. 
\end{lemma}	
\begin{proof}
	We have
	\begin{equation}
	\sup_{h\in \T^3}\sup_{t\in [0,T]}\left|\E \int_{\T^3} u^{i}(t,x) u^{j}(t,x+h) \dee x\right|\leq \E \norm{u}_{L^\infty_t L^2_x}^2,
	\end{equation}
	which is bounded by \eqref{eq:apriori-bounds}. 
	The continuity follows because $u\in L^q_\omega L^\infty_t L^2_x$ for any $q\in [1,\infty)$ by \eqref{eq:apriori-bounds} and shifts of $L_x^p$ functions are continuous. 
	To see that the time averages are twice continuously differentiable, we assume for the moment that $\Gamma_{ij}$ is smooth (e.g. by convolving it with a mollifier) and differentiate, giving
	\begin{align}
	\left|\frac{1}{T}\int_0^T \Grad^2_{h}\Gamma_{ij}(t,h) \dee t\right|& 
	= \left|\frac{1}{T}\int_0^T\E\int_{\T^3} u^{i}(t,x) \Grad^2_h u^{j}(t,x+h) \dee x \dee t\right|\notag\\
	&=\left|\frac{1}{T}\int_0^T\E\int_{\T^3} u^{i}(t,x) \Grad_x^2 u^{j}(t,x+h) \dee x \dee t\right|\notag\\
	&=\left|\frac{1}{T}\int_0^T\E\int_{\T^3} \Grad_xu^{i}(t,x)  \Grad_x u^{j}(t,x+h) \dee x \dee t\right|\label{ineq:gradGam}\\
	&\leq \frac{1}{T} \E \norm{\Grad u}_{L^2_{t,x}}^2\notag. 
	\end{align}
	By approximation, this holds for any function that satisfies $\E \norm{\Grad u}_{L^2_{t,x}}^2<\infty$ which is the case by~\eqref{eq:apriori-bounds}. The continuity of 
	\begin{align}
	h\mapsto\frac{1}{T}\int_0^T\Grad^2_h \Gamma_{ij}(t,h) \dee t
	\end{align} 
	follows from the continuity of shifts of $\Grad_x u$ in $L^2_{\omega,t,x}$ and the third line \eqref{ineq:gradGam}. 
	Similarly, one shows boundedness and continuity of $\Grad_h \Gamma_{ij}$.
	For time-stationary functions, the same results hold without averaging in time. 
\end{proof}

\section{The K\'arm\'an-Howarth-Monin relation}
The goal of this section is to derive the K\'arm\'an-Howarth-Monin (KHM) relation for martingale solutions to \eqref{eq:NSE}. Generally speaking, it is a relation between the two point correlation matrix for the velocity fields
\begin{align}
\Gamma(t,h) = \E \int_{\T^3} u(t,x)\tensor u(t,x+h)\,\dx, 
\end{align}
and for each $k$, the third order structure matrix
\begin{align}
D^k(t,h) = \E \int_{\T^3} \left(\delta_{h}u(t,x)\tensor\delta_h u(t,x)\right)\,\delta_h u^k(t,x)\,\dx.
\end{align}
  \begin{proposition}[KHM relation]\label{prop:KHM}
    Let $u$ be a martingale solution to \eqref{eq:NSE} on $[0,T]$. Let $\eta(h) = (\eta_{ij}(h))_{ij = 1}^3$ be a smooth, compactly supported, isotropic, rank $2$ test function of the form
    \begin{align}\label{eq:isotropic-test}
    \eta(h) = \phi(|h|) I  + \varphi(|h|)\hat{h}\tensor \hat{h}, \qquad \hat{h} = \frac{h}{|h|},
        \end{align}
        where $\phi(\ell)$ and $\varphi(\ell)$ are smooth and compactly supported on $(0,\infty)$. Then the following equality holds
    \begin{align}\label{eq:KHM-gen}
    \begin{split}
    	 &\int_{\R^3}\eta(h) : \Gamma(T,h)\, \dee h - \int_{\R^3}\eta(h) : \Gamma(0, h)\, \dee h =  -\frac{1}{2}\sum_{k=1}^3\int_{0}^{T}\int_{\R^3}\partial_{k}\eta( h) : D^{k}(t, h)\,\dee h\dt\\
    	&\hspace{1in}  + 2\nu \int_{0}^T\int_{\R^3}\Delta \eta( h):\Gamma(t, h)\,\dee h\dt + 2T\int_{\R^3}\eta( h):a( h)\dee h,
    \end{split}
    \end{align}
    where $a( h)$ is related to the two point covariance function associated with the noise by
    \begin{align}\label{eq:A}
    a( h) = \frac{1}{2}\int_{\T^3} C(x,x+ h)\dx,\quad C(x,y) = \sum_{k} \sigma_k^2 e_k(x)\tensor e_k(y).
    \end{align}
  \end{proposition}
 \begin{remark}
 By Cauchy-Schwarz inequality and the assumption $\eps < \infty$, $C(x,x+h)$ is in $L^1_x$ uniformly in $h$; in particular $\abs{a(h)} \lesssim \eps$ and by continuity of shifts in $L^1$, $a(h)$ is a continuous function of $h$. 
 Note that $\tr a(0) = \eps$. 
  \end{remark} 
 \begin{remark} 
 Note that homogeneity of $W$ in space is equivalent to the existence of some continuous $Q$ such that $C(x,y) = Q(x-y)$ for all $x,y \in \Torus^3$.  
\end{remark} 
 
  A version of the above relation (under the assumption of isotropy for the solutions) was first derived by K\'arm\'an and Howarth~\cite{deKarman1938}, and later put into more general form by Monin~\cite{MoninYaglom}. Note that we do not require any more regularity than that possessed by all martingale solutions (see Proposition~\ref{prop:apriori}) and \eqref{eq:KHM-gen} holds for a general isotropic tensor $\eta$. We will derive both the 4/3 and the 4/5 law using this generalized version of the KHM-relation. 
 
 In the proof of Proposition~\ref{prop:KHM}, we will use the following lemma that relates fluxes in the evolution of the two point correlation matrix $\Gamma(t, h)$ to the third order structure matrices $D^k(t, h)$. This plays a fundamental role in the K41 theory.
\begin{lemma}\label{lem:gen-khm-iden}
Let $u\in L^3(\T^3)$ be a divergence-free vector field, and let $T_h u(x) = u(x+h)$ denote its shift by $h\in \R^3$, then the following identity holds in the sense of distributions
 \begin{align}
  &\sum_{k=1}^3\partial_{h_k}\int_{\T^3}  \big[(u \tensor T_h u)u^k  + (T_h u \tensor u)u^k - (T_h u \tensor u)T_h u^k  - (u \tensor T_h u) T_h u^k \big]\dx\notag\\
  &\hspace{1in} = \sum_{k=1}^3\partial_{h_k}\int_{\T^3} (\delta_h u\tensor \delta_h u) \delta_h u^k\, \dx.
 \end{align}
\end{lemma}
\begin{proof}

The proof is mostly algebraic and is most easily seen by working backwards. Expanding the differences we find,
  \begin{align}
   \int (\delta_h u\tensor \delta_h u) \delta_h u^k\, \dx & =  \int (T_h u \tensor T_h u)T_h u^k \, \dx - \int (u \tensor T_h u)T_h u^k \, \dx - \int (T_h u \tensor T_h u)u^k \, \dx \notag\\
                & \quad- \int (T_h u \tensor u)T_h u^k \, \dx + \int (u\tensor T_h u)u^k \, \dx + \int (u\tensor u)T_h u^k \, \dx\notag\\
                & \quad+ \int (T_h u\tensor u) u^k\, \dx - \int (u\tensor u) u^k\, \dx.
  \end{align}
 Using the fact that
 \begin{align}
  \int (T_h u\tensor T_h u)\,T_h u^k\, \dx = \int (u\tensor u) u^k\, \dx,
 \end{align}
 and the divergence free property of $u$,
 \begin{align}
  \sum_{k=1}^3\partial_{h_k}\int (T_h u\tensor T_h u)\,u^k \, \dx =  \sum_{k=1}^3\partial_{h_k}\int (u\tensor u)
  \,T_h u^k \, \dx = \int (u\tensor u)
  \,T_h (\Div u^k) \, \dx = 0,
 \end{align}
 we obtain
 \begin{align}
   &\sum_{k=1}^3\partial_{h_k}\int (\delta_h u\tensor \delta_h u) \delta_h u^k\, \dx =  - \sum_{k=1}^3\partial_{h_k}\int (u \tensor T_h u)\, T_h u^k\, \dx - \partial_{h_k}\int (T_h u\tensor u)\,T_h u^k\, \dx\notag\\
   			&\hspace{1in} + \sum_{k=1}^3\partial_{h_k}\int (u\tensor T_h u)\, u^k\, \dx + \partial_{h_k}\int  (T_h u\tensor u) u^k\, \dx\notag\\
            &\hspace{.5in} =\sum_{k=1}^3\partial_{h_k}\int  \big[(u \tensor T_h u)u^k  + (T_h u \tensor u)u^k - (T_h u \tensor u)T_h u^k  - (u \tensor T_h u) T_h u^k \big]\dx
 \end{align}
 completing the proof.
\end{proof}

We are now ready to prove the proposition.
\begin{proof}[Proof of Proposition \ref{prop:KHM}]
    Let $\gamma(x)$ be a standard bump (smooth, radially symmetric, positive, supported in the ball with unit integral), and for each $\kappa >0$ define $\gamma_\kappa(x) = \kappa^{-3}\gamma(\kappa^{-1}x)$. For any function $f(x)$ on $\T^3$ we denote the spatially mollified function by
    \begin{equation}
      f_\kappa = (f)_\kappa := \gamma_\kappa\star f.
    \end{equation}
    The mollified stochastic equation for $u_{\kappa} = \gamma_\kappa\star u$ becomes for each $x\in \T^3$
    \begin{equation}\label{eq:mollified-eq}
      \dee u_\kappa(t,x) + \Div_x (u\tensor u)_\kappa(t,x)\dt + \nabla p_\kappa(t,x)\dt - \nu \Delta u_\kappa(t,x)\dt = \dee W_\kappa(t,x),
    \end{equation}
    where the equation is to be interpreted as a finite dimensional SDE for each $x\in \T^3$ and $\kappa >0$. Let $h \in \T^3$, then the evolution of $u_\kappa(t,x)\tensor u_\kappa(t,x+h)$ satisfies the stochastic product rule
    \begin{align}\label{eq:stochastic-prod-rule}
    \begin{split}
    	  \dee (u_{\kappa}(t,x)\tensor u_{\kappa}(t,x+h)) &= \dee u_{\kappa}(t,x)\tensor u_{\kappa}(t,x+h) + u_{\kappa}(t,x)\tensor \dee u_{\kappa}(t,x+h)\\
    	&\quad+ \dee[u_{\kappa}(\cdot,x),u_{\kappa}(\cdot,x+h)](t)
    \end{split}
    \end{align}
    where $[u_{\kappa}(\cdot,x),u_{\kappa}(\cdot,x+h)](t)$ is the cross variation of $u_{\kappa}(\cdot,x)$ and $u_{\kappa}(\cdot,x+h)$ and is given by
      \begin{align}
       \big[u_{\kappa}(\cdot,x),u_{\kappa}(\cdot,x+h)\big](t) &= \big[W_{\kappa}(\cdot,x), W_{\kappa}(\cdot,x+h)\big](t)\notag\\
        &=\sum_{k,m} \sigma_k\sigma_m \left(e_{\kappa}^k(x)\tensor e_{\kappa}^m(x+h)\right)\big[\beta_k,\beta_m\big](t)\notag\\
        &= t\,C^{\kappa}(x,x+h),
      \end{align}
    with
    \begin{align}
      C^\kappa(x,x+h) = \sum_{j=1}^\infty \sigma_j^2\, e_{\kappa}^j(x)\tensor e_{\kappa}^j(x+h).
    \end{align}
    Upon integrating equation~\eqref{eq:stochastic-prod-rule} in $x$, using equation~\eqref{eq:mollified-eq} and integrating by parts, we readily obtain the following matrix valued SDE for each $h\in \R^3$
      \begin{align}
        &\dee\left(\int_{\T^3} u_{\kappa}\tensor T_h u_{\kappa}\, \dx\right) = \sum_{k=1}^3\left(\partial_{h_k}\int_{\T^3}\left[(u^k u)_\kappa \tensor T_h u_{\kappa} - u_{\kappa}\tensor (T_h u^k\,T_h u)_\kappa)\right]\,\dx\right)\dt\notag\\
        &\hspace{.5in}+ 2\nu \left(\Delta_h\int_{\T^3}u_{\kappa}\tensor T_h u_{\kappa}\,\dx\right)\dt + \left(\nabla_h\int_{\T^3}(p_\kappa\, T_h u_{\kappa})\,\dx - \nabla^\top_h\int_{\T^3}(u_{\kappa}\,T_h p_\kappa)\,\dx\right)\dt\notag\\
         &\hspace{.5in}+ \left(\int_{\T ^3}C^\kappa(x,x+h)\dx\right)\dt + \sum_{j=1}^\infty\left(\int_{\T^3}\sigma_j \left[e_{\kappa}^j\tensor T_h u_{\kappa} + u_{\kappa}\tensor T_h e_{\kappa}^j\right]\right)\dee \beta_j,
      \end{align}
      where we used $\Grad^\top_h$ to denote the transpose of $\Grad_h$, that is, $(\Grad_h^\top v)_{ij}=(\Grad_h v)_{ji}=\partial_{h_j} v^i$ for a vector $v=(v^1,v^2,v^3)^\top$. 
    Integrating in time, taking expectation and pairing against the isotropic test function $\eta$ (defined by equation \eqref{eq:isotropic-test}) in the $h$ variable and integrating by parts gives
    \begin{align}\label{eq:ep-eq}
    \begin{split}
      &\int_{\R^3}\eta(h) : \Gamma^\kappa(T,h)\, \dee h - \int_{\R^3}\eta(x) : \Gamma^\kappa(0,h)\, \dee h =  -\frac{1}{2}\int_{0}^{T}\int_{\R^3}\partial_{k}\eta(h) : D^{\kappa,k}(t,h)\,\dee h\dt\\
      &\hspace{1in}  + 2\nu \int_{0}^T\int_{\R^3}\Delta \eta(h):\Gamma^\kappa(t,h)\,\dee h\dt + \int_0^T\int_{R^3} \eta(h) : P^\kappa(t,h)\dee h\dt\\
      &\hspace{1in} + 2T\int_{\R^3}\eta(h):a^\kappa(h)\dee h,
      \end{split}
    \end{align}
      where we have defined the ``regularized'' quantities
      \begin{align}
        \Gamma^\kappa(t,h) := \E\int_{\T^3} u_\kappa(t,x)\tensor u_\kappa(t,x+h)\dx,
      \end{align}
      \begin{align}
        D^{\kappa,k}(t,h) := 2\E\int_{\T^3} \left[(u^k u)_\kappa(t,x)\tensor u_\kappa(t,x+h) - u_\kappa(t,x)\tensor(u^k u)_\kappa(t,x+h))\right]\,\dx,
      \end{align}
      \begin{align}
      a^\kappa(h) := \frac{1}{2}\int_{\T^3} C^\kappa(x,x+h)\dx,
      \end{align}
      and
      \begin{align}
      P^\kappa(t,h) := \E\left(\nabla_h\int_{\T^3}p_\kappa(t,x)u_\kappa(t,x+h)\dx - \nabla^\top_h\int_{\T^3} u_\kappa(t,x)p_\kappa(t,x+h)\dx\right).
      \end{align}
    Ultimately, the goal is to pass $\kappa \to 0$ in \eqref{eq:ep-eq} to obtain equation \eqref{eq:KHM-gen}. However, before sending $\kappa \to 0$, we must deal with the pressure term $P^\kappa$. We will see that since $u_\kappa$ is divergence free and we are pairing with an isotropic test function $\eta$, the contribution from the pressure, $P^\kappa$, vanishes
    \begin{align}\label{eq:pressure-vanish}
    \int_{\R^3} \eta(h):P^\kappa(t,h)\dee h = 0.
    \end{align}
    To show this, we observe that, by symmetry of $\eta$ and integration by parts in $h$, and using the identity,
    \begin{align}
    \Div_h\eta(h) = \Div_h\left(\phi(|h|)I + \varphi(|h|)\hat{h}\tensor\hat{h}\right) = \Phi(|h|)\hat{h},\qquad \Phi(\ell) = \phi^\prime(\ell) + \varphi^\prime(\ell) + 2\ell^{-1}\varphi(\ell),
    \end{align}
    we have
    \begin{align}
   \int_{\R^3} &\eta(h):P^\kappa(t,h)\dee h\notag\\
   &= -\E\iint \Div_h\eta(h)\cdot p_\kappa(t,x)\left(u_\kappa(t,x+h) - u_\kappa(t,x-h) \right)\dx \dee h\notag\\
   &= -\E\iint \Phi(|h|)\,p_\kappa(t,x)\left(u_\kappa(t,x+h) - u_\kappa(t,x-h) \right)\cdot \hat{h}\,\dx \dee h\notag\\
   &= -\E\int p_\kappa(t,x)\left(\int_{\R_+}\Phi(\ell)\left[\int_{|h| = \ell}(u_\kappa(t,x+h) - u_\kappa(t,x-h))\cdot\hat{h}\, \dee S(h)\right] \dee \ell \right)\dx.
\end{align}
By the divergence theorem, we conclude that
\begin{align}
 \int_{|h| = \ell}(u_\kappa(x+h) - u_\kappa(x-h))\cdot\hat{h}\, \dee S(h) = \int_{ |h|\leq \ell} 2 \Div u_\kappa(x+h)\dee h =0,
\end{align}
and therefore \eqref{eq:pressure-vanish} holds for all suitably regular isotropic tensors $\eta(h)$.

Now we aim to pass to the limit as $\kappa \to 0$. It is a simple consequence of the form of the noise that the convergence $C^\kappa(x,x+h) \to C(x,x+h)$ holds locally uniformly in both $x$ and $h$ and therefore $a^\kappa(h) \to a(h)$ locally uniformly. Additionally, using the standard properties of mollifiers on $L^2$ vector fields, we have that $\P \tensor \dt$ almost everywhere in $\Omega \times [0, T]$
    \begin{align}
      u_\kappa \to u\quad\text{in}\quad L^2_x.
    \end{align}
    The moment bound $\sup_{t\in [0,T]}\E\|u_\kappa(t)\|_{L^2_x}^p <\infty$ for any $p > 2$ implies that for each $t$, the sequence $(u_{\kappa}(t))_{\kappa > 0}$ is uniformly integrable in $L^2(\Omega\times\T^3)$. Therefore, by the Vitali convergence Theorem, for each $(t,h) \in [0,T]\times\R^3$,
    \begin{align}\label{eq:pointwise-Gamma}
      \Gamma^\kappa(t,h) \to \Gamma(t,h).
    \end{align}
    Furthermore the uniform bound
    \begin{align}
      \Gamma^\kappa(t,h) \leq \E\|u\|_{L^\infty_tL^2_x}^2 <\infty,
    \end{align}
    and the bounded convergence theorem imply
    \begin{align}\label{eq:L1-Gamma}
    \Gamma^\kappa \to \Gamma \quad \text{in}\quad L^1_{\loc}([0,T]\times\R^3).
    \end{align}
    Both the pointwise convergence \eqref{eq:pointwise-Gamma} and $L^1$ convergence \eqref{eq:L1-Gamma} are enough to pass to the limit in all terms of equation \eqref{eq:ep-eq} that involve $\Gamma^\kappa$.

    It remains to pass to the limit in the term involving $D^{\kappa,k}$. Using the properties of mollifiers and the fact that $u\in L^3(\Omega\times[0,T]\times\T^3)$, we conclude that for $\P\tensor\dt\tensor \dee h$ almost every $(\omega,t,h)\in\Omega\times [0,T]\times\R^3$, both
    \begin{align}\label{eq:pointwise-each-term}
    &(u^k\,u)_\kappa(\omega,t,\cdot)\tensor T_h u_\kappa(\omega,t,\cdot)\to (u\tensor T_h u )(\omega,t,\cdot)u^k(\omega,t,\cdot),\notag\\
    \intertext{and}
    &u_\kappa(\omega,t,\cdot)\tensor T_h(u^k\,u)_\kappa(\omega,t,\cdot)\to (u\tensor T_h u )(\omega,t,\cdot)T_h u^k(\omega,t,\cdot),
    \end{align}
    converge in $L^1(\T^3)$. Additionally, the fact that $u \in L^q(\Omega\times[0,T]\times\T^3)$ for each $2\leq q < 10/3$, implies that 
    for each $h \neq 0$, the sequences $((u^k\,u)_\kappa\tensor T_h u)_{\kappa> 0}$ and $(u_\kappa\tensor T_h(u^k\,u)_\kappa(\omega,t,\cdot))_{\kappa>0}$ are also uniformly integrable in $L^1(\Omega\times[0,T]\times \T^3)$. Combining this with the pointwise convergence \eqref{eq:pointwise-each-term}, the Vitali convergence Theorem implies that for each $h$ the following convergence holds in $L^1([0,T])$
    \begin{align}
    D^{\kappa,k}(\cdot,h) \to 2\E\int_{\T^3} \left[(u(\cdot,x)\tensor T_h u(\cdot,x))\, u^k(\cdot,x) - (u(\cdot,x)\tensor T_h u(\cdot,x)) \,T_h u^k(\cdot,x)\right]\,\dx.
    \end{align}
    Using the symmetry of $\eta$ and the identity from Lemma \ref{lem:gen-khm-iden} we find that the limit can be written in terms of the third order structure function
    \begin{align}
    &2\int_{\R^3}\left(\partial_{h_k}\eta(h):\E\int_{\T^3} \left[(u(t,x)\tensor T_h u(t,x))\, u^k(t,x) - (u(t,x)\tensor T_h u(t,x)) \,T_h u^k(t,x)\right]\,\dx\right)\dee h\notag\\
     &\hspace{1in}= \int_{\R^3}\partial_{h_k}\eta(h):D^k(t,h)\dee h.
    \end{align}
    Finally, using the uniform bound
    \begin{align}
    \abs{\partial_{h_k}\eta(h):D^{\kappa,k}(t,h)} \leq \E\|u\|_{L^3_{t,x}}^3<\infty,
    \end{align}
    it follows from the bounded convergence theorem that
    \begin{align}
    \partial_{h_k}\eta:D^{\kappa,k}\to \partial_{h_k}\eta:D^{k} \quad \text{in}\quad L^1([0,T]\times\R^3)
    \end{align}
    and therefore we may pass the limit of the final term in equation \eqref{eq:ep-eq}.
  \end{proof}

  Naturally the same proof, with some small modifications, can be applied to the case of stationary martingale solutions to \eqref{eq:NSE}, in which case the KHM relation takes a simpler form.
  \begin{corollary}\label{cor:stationary-KHM}
  Let $u$ be a stationary weak martingale solution to \eqref{eq:NSE}. Let $\eta(h) = (\eta_{ij}(h))_{ij = 1}^3$ be a smooth, compactly supported, isotropic, rank two test function of the form
    \begin{align}\label{eq:testfun}
    \eta(h) = \phi(|h|) I  + \varphi(|h|)\hat{h}\tensor \hat{h}, \qquad \hat{h} = \frac{h}{|h|},
        \end{align}
        where $\phi(\ell)$ and $\varphi(\ell)$ are smooth and compactly supported on $(0,\infty)$. Then the following equality holds,
    \begin{align}\label{eq:KHM-gen-stationarry}
      \frac{1}{2}\sum_{k=1}^3\int_{\R^3}\partial_{h_k}\eta(h) : D^{k}(h)\,\dee h=2\nu \int_{\R^3}\Delta \eta(h):\Gamma(h)\,\dee h + 2\int_{\R^3}\eta(h):a(h)\,\dee h.
    \end{align}
  \end{corollary}

\section{Proof of the 4/3 law}
In this section we first prove the (spherically averaged) 4/3 law as stated in \eqref{43Law}, and then, using similar ideas, we prove  Proposition \ref{prop:Triv}.

\subsection{Proof of (\ref{43Law})}
From \eqref{eq:GAMMA} and \eqref{eq:A}, we define the spherical averages 
\begin{align}\label{eq:averageGAM}
 \bar{\Gamma}(\ell)  = \frac{1}{4\pi}\int_{\S^2} I :\Gamma(\ell \hat{n}) \dee S(\hat{n}), \qquad \bar{a}(\ell)  = \frac{1}{4\pi}\int_{\S^2}\tr a(\ell \hat{n}) \dee S(\hat{n}). 
 \end{align} 
 Notice that $\Gamma$ does not depend on time since we are considering stationary solutions.
From Corollary~\ref{cor:stationary-KHM} applied with only the test function $\phi$ in \eqref{eq:testfun} and Lemma~\ref{lem:gammaref},
we deduce that
 \begin{align}
  \int_{\Real^+} \ell^2 \phi(\ell)  \left(\nu \bar{\Gamma}'' + \nu \frac{2}{\ell} \bar{\Gamma}'  + \bar{a}(\ell) \right) \dee\ell = \frac{1}{4}\int_{\Real^+} S_0(\ell) \ell^2 \phi'(\ell)  \dee\ell.
 \end{align}
 Integration by parts gives
 \begin{align}
  \frac{1}{4}\int_{\Real^+}  S_0(\ell) \ell^2 \phi'(\ell)  \dee\ell & = -\lim_{\ell \rightarrow 0}\frac{1}{4} S_0(\ell) \ell^2 \phi(\ell) - \frac{1}{4}\int_{\Real^+} \left(S_0'(\ell) + \frac{2}{\ell}S_0(\ell)\right) \ell^2 \phi(\ell) \dee\ell\notag \\ 
  & = - \frac{1}{4}\int_{\Real^+} \left(S_0'(\ell) + \frac{2}{\ell}S_0(\ell)\right) \ell^2 \phi(\ell) \dee\ell, 
 \end{align}
 where the boundary term vanishes due to Lemma \ref{lem:S0cont}. 
 Hence, the following ODE holds in the sense of distribution
 \begin{align}
  -\frac{\ell^2}{4}\left(S_0' + \frac{2}{\ell}S_0\right) = \ell^2 \left(\nu \bar{\Gamma}'' + \nu \frac{2}{\ell} \bar{\Gamma}' + \bar{a}\right). 
 \end{align}
Note that this implies $S_0$ is differentiable for all $\ell \in (0,1/2)$. 
We may re-write this as 
 \begin{align}
  \partial_\ell \left(\ell^3 \frac{S_0}{\ell}\right) = -\ell^2 \left( 4\nu \bar{\Gamma}'' + 4\nu \frac{2}{\ell} \bar{\Gamma}' + 4\bar{a} \right). 
 \end{align}
Recall from Lemma \ref{lem:gammaref} that $\Gamma \in C_\ell^{2}$ for each fixed $\nu>0$, 
and hence the right-hand side is not  singular. Therefore, integration yields: 
 \begin{align}
  \frac{S_0(\ell)}{\ell} = -\frac{1}{\ell^3}\int_0^\ell \tau^2 \left( 4\nu \bar{\Gamma}''(\tau) + 4\nu \frac{2}{\tau} \bar{\Gamma}'(\tau) + 4\bar{a}(\tau) \right) \dee\tau.\label{eq:S0F}
 \end{align}
Firstly, since
 \begin{align}
  \frac{4}{\ell^3}\int_0^\ell \tau^2 \bar{a}(\ell) \dee\tau & = \frac{4}{3}\bar{a}(0) + \frac{4}{\ell^3}\int_0^\ell \tau^2\left(\bar{a}(\tau) - \bar{a}(0)\right) \dee\tau,\label{eq:23a0}
 \end{align}
 we use the continuity of $\bar{a}$ and infer that
 \begin{align}
  \lim_{\ell_I \rightarrow 0} \sup_{\ell \in (0,\ell_I)} \frac{4}{\ell^3}\int_0^\ell \tau^2\left(\bar{a}(\tau) - \bar{a}(0)\right) \dee\tau = 0. \label{eq:acont}
 \end{align}
Moreover, a further integration by parts gives 
 \begin{align}
  \frac{1}{\ell^3}\int_0^\ell\left[ \tau^2 \bar{\Gamma}''(\tau) + 2\tau \bar{\Gamma}'(\tau) \right]\dee\tau & = \frac{\bar\Gamma'(\ell)}{\ell}. \label{eq:IBPGam}
 \end{align}
  Recalling that 
 \begin{align}
 \grad_h \tr \Gamma (h)= \EE\int_{\T^3} u(t,x)  \grad_h u(t,x+h) \dx = - \EE\int_{\T^3}u(t,x+h) \grad_x u(t,x)   \dx 
 \end{align} 
we obtain 
\begin{align}
\abs{\bar{\Gamma}'(\ell)} \lesssim \ell \,\EE\norm{\grad u}_{L^2_x}^2. 
\end{align} 
 Hence, recalling also that $\bar{a}(0) = \eps$, \eqref{eq:S0F} becomes 
  \begin{align}\label{eq:43balancethingy}
  \frac{S_0(\ell)}{\ell} = -\frac{4\nu \bar{\Gamma}'(\ell)}{\ell} -\frac{4}{3}\eps +o_{\ell\to 0}(1)
 \end{align}
 where the $o_{\ell \to 0}(1)$ vanishes as $\ell \to 0$ \emph{uniformly} with respect to $\nu$. 
Next, using \eqref{eq:EnergyIneq},
 \begin{align}
 \abs{\frac{\nu}{\ell}\bar\Gamma'(\ell)} \lesssim \frac{\nu}{\ell}\left(\EE \norm{\grad u}^2_{L^2_x}\right)^{1/2}\left(\EE \norm{u}_{L^2_x}^2\right)^{1/2}
 \lesssim \frac{(\eps\nu)^{1/2}}{\ell}\left(\EE \norm{u}_{L^2_x}^2\right)^{1/2}.
 \end{align}
Therefore, by the weak anomalous dissipation condition \eqref{eq:WAD}, we can choose $\ell_D(\nu) \rightarrow 0$ such that $\left(\nu \EE \norm{u}^2\right)^{1/2} = o(\ell_D)$. Hence, for all $\ell_I < 1/2$, there holds
 \begin{align}
  \lim_{\nu \rightarrow 0} \sup_{\ell \in (\ell_D,\ell_I)} \abs{\frac{\nu}{\ell}\bar\Gamma'(\ell)} = 0. \label{eq:GamPrimeVanish} 
 \end{align}
 The result stated in \eqref{43Law} follows after applying \eqref{eq:GamPrimeVanish} and \eqref{eq:43balancethingy}.  The proof is concluded.

\subsection{Proof of Proposition \ref{prop:Triv}}
The above proof is easily adapted to prove Proposition~\ref{prop:Triv}.
 Consider  the general case in which $u$ does not obey anomalous dissipation.
Recall from \eqref{eq:S0F}  that
 \begin{align}
  \frac{S_0(\ell)}{\ell} & = -\nu \frac{1}{\ell^3}\int_0^\ell \tau^2 4\left(\bar{\Gamma}'' + \frac{2}{\ell}\bar{\Gamma}' \right) \dee\tau - \frac{4}{\ell^3} \int_0^\ell \tau^2 \bar{a}(\tau) \dee\tau. \label{eq:S0ellProp}
 \end{align}
Furthermore, from the definition of $\bar\Gamma$ in \eqref{eq:averageGAM}, we have that
 \begin{align}
  \bar{\Gamma}''(\ell) + \frac{2}{\ell}\bar{\Gamma}'(\ell) = -\frac{1}{4\pi}\int_{\S^2}\int_{\T^3}\EE \grad u(x) : \grad u(x + \ell\hat{n})\dee x \dee S(\hat{n}).
 \end{align}
 For each fixed $\nu$, the energy inequality \eqref{eq:EnergyIneq} implies
 \begin{align}
\nu \abs{\bar{\Gamma}'' + \frac{2}{\ell} \bar{\Gamma}'} \leq \nu \EE \|\grad u\|_{L^2_x}^2 \leq \bar{a}(0). 
 \end{align}
Therefore, from \eqref{eq:S0ellProp}, 
 \begin{align}
 -\frac{8}{3}\bar{a}(0) - \frac{4}{\ell^3}\int_0^\ell \tau^2\left(\bar{a}(\tau) - \bar{a}(0)\right) \dee\tau \leq \frac{S_0(\ell)}{\ell} & \leq -\frac{4}{\ell^3}\int_0^\ell \tau^2\left(\bar{a}(\tau) - \bar{a}(0)\right) \dee\tau. 
 \end{align}
 Proposition \ref{prop:Triv} then follows.

\section{Proof of the 4/5 law}
We conclude here the proof of our main Theorem \ref{thm:T}, by showing the validity of \eqref{45Law}.
Our starting point will be the stationary form of the KHM relation \eqref{eq:KHM-gen-stationarry} where we take the isotropic test function $\eta$ to be of the form
\begin{equation}
\eta(h) = \varphi(|h|)\hat{h}\tensor\hat{h}, \qquad \hat{h} = \frac{h}{|h|},
\end{equation}
i.e., we take $\phi(|h|) = 0$ in the general form \eqref{eq:isotropic-test}. This gives the following balance relation
\begin{equation}\label{eq:stat-KHM-1}
\sum_{k=1}^3\int_{\R^3}\partial_{h_k}(\varphi(|h|)\hat{h}\,\tensor\hat{h}) : \left(\frac{1}{2}D^{k}(h)+ 2\nu\partial_{h_k}\Gamma(h)\right)\,\dee h = 2\int_{\R^3}\varphi(|h|)\,\hat{h}\tensor\hat{h}:a(h)\dee h.
\end{equation}
Using the following elementary identity
\begin{equation}\label{eq:phill-relation}
 \partial_{h_k} (\varphi(|h|)\hat{h}\tensor \hat{h}) = \left(\varphi^\prime(|h|)- 2|h|^{-1}\varphi(|h|)\right)\hat{h}\tensor\hat{h}\, \hat{h}_k + |h|^{-1}\varphi(|h|)\left(e^k\tensor\hat{h} + \hat{h}\tensor e^k\right)
\end{equation}
we may write
\begin{equation}\label{eq:stat-khm1}
 \begin{aligned}
 &\frac{1}{2}\sum_{k=1}^3\int_{\R^3}\partial_{h_k}(\varphi(|h|)\hat{h}\,\tensor\hat{h}) : D^{k}(h)\,\dee h = \frac{1}{2}\E\iint_{\T^3\times\R^3}\left(\varphi^\prime(|h|)- 2|h|^{-1}\varphi(|h|)\right)(\delta_h u\cdot\hat{h})^3\dx\dee h\\
 & \hspace{1in}+ \E\iint_{\T^3\times\R^3}|h|^{-1}\varphi(|h|)|\delta_h u|^2(\delta_h u\cdot\hat{h})\dx\dee h\\
 &\hspace{.5in} = 2\pi\int_{\R_+} \left(\ell^2\varphi^\prime(\ell) - 2\ell\varphi(\ell)\right)S_{||}(\ell)\,\dee\ell + 4\pi\int_{\R_+} \ell\varphi(\ell)\,S_{0}(\ell)\,\dee\ell
 \end{aligned}
\end{equation}
and after an integration by parts, write
\begin{equation}\label{eq:stat-khm2}
 \begin{aligned}
 &\sum_{k=1}^3\int_{\R^3}\partial_{h_k}(\varphi(|h|)\hat{h}\,\tensor\hat{h}) : 2\nu\partial_{h_k} \Gamma(h)\,\dee h\\
 & \hspace{.5in}= 2\nu\E\iint_{\T^3\times\R^3}\left(\varphi^\prime(|h|)- 2|h|^{-1}\varphi(|h|)\right)(\hat{h}\cdot u)T_h(\hat{h}\tensor\hat{h}:\nabla u)\dx\dee h\\
 & \hspace{1in}+ 4\nu\E\iint_{\T^3\times\R^3}|h|^{-1}\varphi(|h|)\left((\hat{h}\cdot u)T_h (\Div u) - (\Div u)(\hat{h}\cdot T_h u)\right)\dx\dee h\\
 &\hspace{.5in} = 8\pi \nu\int_{\R_+}(\ell^2\varphi^\prime(\ell)- 2\ell\varphi(\ell)) H(\ell)\dee\ell.
 \end{aligned}
\end{equation}
where we have defined $H(\ell)$ by
\begin{equation}\label{eq:stat-khm3}
H(\ell) = \frac{1}{4\pi}\E\int_{\S^2}\left(\int_{\T^3}(\hat{n}\cdot u)(\hat{n}\tensor \hat{n}:T_{\hat{n}\ell}\nabla u)\,\dx\right)\dee S(\hat{n}).
\end{equation}
Denoting 
\begin{equation}
\tilde{a}(\ell)=\frac{1}{4\pi}\int_{\S^2}\hat{n}\tensor\hat{n}:a(\ell \hat{n}) \dee S(\hat{n}),
\end{equation}
we obtain
\begin{equation}
2\int \varphi(|h|)\hat{h}\tensor\hat{h}:a(h)\dee h = 8\pi\int_{\R_+} \ell^2\varphi(\ell)\tilde{a}(\ell)\dee \ell.
\end{equation}
Using the fact that $\ell^2\varphi^\prime(\ell) - 2\ell\varphi(\ell) = \ell^4\left(\ell^{-2}\varphi(\ell)\right)'$ and collecting the identities \eqref{eq:stat-khm1},\eqref{eq:stat-khm2}, and \eqref{eq:stat-khm3} we can write the balance relation \eqref{eq:stat-KHM-1} as the following distributional ODE\\
\begin{equation}
\partial_\ell\left[\ell^4 (S_{||}(\ell) + 4\nu H(\ell))\right] - 2\ell^3 S_0(\ell) + 4 \ell^{4}\tilde{a}(\ell) = 0 \quad \text{in}\quad \mathcal{D}^\prime(0,\infty).
\end{equation}
Integrating both sides of this equation gives
\begin{equation}\label{eq:integrated-ode}
\ell^4 (S_{||}(\ell) + 4\nu H(\ell)) = \lim_{\ell\to 0+} \left(\ell^4 (S_{||}(\ell) + 4\nu H(\ell))\right) + 2\int_0^\ell \tau^3 S_0(\tau)\dee\tau- 4\int_0^\ell \tau^{4} \tilde{a}(\tau)\dee\tau.
\end{equation}
The boundary term vanishes thanks to Lemma \ref{lem:S0cont} and the fact that $\nu H(\ell)$ is bounded by the energy dissipation
\begin{equation}
\sup_{\ell\in \R_+}\nu H(\ell) \leq \left(\nu\E\|u\|_{L^2_x}^2\right)^{1/2}\left(\nu\E\|\nabla u\|_{L^2_x}^2\right)^{1/2} \leq \nu \E\|\nabla u\|_{L^2}^2.
\end{equation}
Dividing both sides of \eqref{eq:integrated-ode} by $\ell^5$ gives 
\begin{equation}\label{eq:sperpoverr}
\frac{S_{||}(\ell)}{\ell} = - 4\nu \frac{H(\ell)}{\ell} + 2\ell^{-5} \int_0^\ell \tau^3 S_0(\tau)\dee\tau- 4\ell^{-5}\int_0^\ell \tau^{4} \tilde{a}(\tau)\dee\tau.
\end{equation}
Under the assumption of weak anomalous dissipation
\begin{equation}
 \nu^{1/2}\left(\E\|u\|^2_{L_x^2}\right)^{1/2} = o(\ell_D),
\end{equation}
we observe that
\begin{equation}\label{eq:lim451}
 \lim_{\ell_I\to 0}\lim_{\nu\to 0}\sup_{\ell\in[\ell_D,\ell_I]}\frac{|H(\ell)|}{\ell} \leq  \lim_{\nu\to 0}\ell_D^{-1}\left(\nu\E\|u\|^2_{L^2}\right)^{1/2}\left(\nu\E\|\nabla u\|_{L^2}^2\right)^{1/2} = 0.
\end{equation}
Moreover, since $a(h)$ is uniformly continuous in $h$ independently of $\nu$ and $\int_{\S^2} \hat{n}\tensor \hat{n} \,\dee S(\hat{n}) = \frac{4}{3}\pi I$, we have that
\begin{equation}
\tilde{a}(\ell)-\frac{1}{3} \tr a(0) = \frac{1}{4\pi}\int_{\S^2} \hat{n}\tensor\hat{n}:(a(\ell \hat{n})-a(0)) \dee S(\hat{n}) = o_{\ell\to 0}(1),
\end{equation}
and therefore,
\begin{equation}\label{eq:lim453}
4\ell^{-5}\int_0^\ell \tau^{4} \tilde{a}(\tau)\dee\tau = \frac{4}{15} \tr a(0) + 4\ell^{-5}\int_0^\ell \tau^{4} \left(\tilde{a}(\tau)-\frac{1}{3} \tr a(0)\right)\dee\tau
= \frac{4}{15} \eps +o_{\ell\to 0}(1)
\end{equation}
where the $o_{\ell\to 0}(1)$ vanishes as $\ell\to 0$ uniformly with respect to $\nu$ and $\tr a(0)=\eps$. 
Then using the $4/3$ law,
\begin{equation}\label{eq:lim452}
 \lim_{\ell_I\to 0}\lim_{\nu\to 0} \sup_{\ell\in[\ell_D,\ell_I]} 2\ell^{-5} \int_0^\ell \tau^3 S_0(\tau)\dee\tau = \frac{2}{5}\left(-\frac{4}{3}\eps\right) = - \frac{8}{15}\eps,
\end{equation}
and combining, the 4/5 law follows by substituting the limits \eqref{eq:lim451}, \eqref{eq:lim453} and \eqref{eq:lim452} into equation \eqref{eq:sperpoverr}.

\section{Necessary conditions for third order scaling laws}
In this section we prove Theorem \ref{thm:Nec}. As for the sufficient condition results of Theorem \ref{thm:T}, it is natural to use 
the result on $S_0$ to deduce the result on $S_{||}$.

\subsection{Proof of (\ref{eq:scalinglaw}) for $S_0$}
From \eqref{eq:43balancethingy},  
\begin{align} 
\frac{S_0(\ell)}{\ell} = - \frac{4\nu \bar{\Gamma}'(\ell)}{\ell} - \frac{4}{3} \eps + o_{\ell \rightarrow 0}(1), \label{eq:s0nec}
\end{align} 
where the $o_{\ell \to 0}(1)$ vanishes as $\ell \to 0$ \emph{uniformly} with respect to $\nu$. 
Then, 
\begin{align} 
\bar{\Gamma}'(\ell) & = \sum_{i,j} \frac{1}{4\pi} \EE \int_{\S^2} \int_{\Torus^3} u^i(x) \partial_{x_j}u^i(x + \ell \hat{n}) \hat{n}^j \dee x \dee S(\hat{n}) \notag\\  
& = - \sum_{i,j} \frac{1}{4\pi} \EE \int_{\S^2} \int_{\Torus^3} \partial_{x_j} u^i(x) u^i(x + \ell \hat{n}) \hat{n}^j \dee x \dee S(\hat{n}). 
\end{align} 
Observe by periodicity that
\begin{align} 
\EE \frac{1}{4\pi} \int_{\S^2} \int_{\Torus^3} \partial_{x_j} u^i(x) u^i(x) \hat{n}^j \dee x \dee S(\hat{n}) = 0, 
\end{align}
and hence
\begin{align} 
\bar{\Gamma}'(\ell) & = -\sum_{i,j} \frac{\ell}{4\pi} \EE \int_{\S^2} \int_{\Torus^3} \partial_{x_j} u^i(x)\left(\frac{u^i(x + \ell \hat{n}) - u^i(x)}{\ell}\right) \hat{n}^j \dee x \dee S(\hat{n}). 
\end{align}
By condition \eqref{eq:condition2} and \eqref{eq:EnergyIneq}
\begin{align} 
\bar{\Gamma}'(\ell) & = -\sum_{i,j} \frac{\ell}{4\pi} \EE \int_{\S^2} \int_{\Torus^3} \partial_{x_j} u^i(x) \partial_{x_k} u^i(x) \hat{n}^k \hat{n}^j \dee x \dee S(\hat{n}) + \frac{\ell}{\nu} o_{\ell \rightarrow 0}(1),
\end{align} 
where crucially, as above, the $o_{\ell \to 0}(1)$ vanishes as $\ell \to 0$ \emph{uniformly} with respect to $\nu$.  
Since (denoting the Kronecker $\delta$ as $\delta_{kj}$)
\begin{align} 
\frac{1}{4\pi}\int_{\S^2} \hat{n}^k \hat{n}^j \dee S(\hat{n}) = \frac{1}{3}\delta_{kj}, 
\end{align} 
we have 
\begin{align} 
\bar{\Gamma}'(\ell) & = - \frac{\ell}{3} \EE \norm{\grad u}_{L^2_x}^2 + \frac{\ell}{\nu} o_{\ell \rightarrow 0}(1).
\end{align} 
In turn, by \eqref{eq:s0nec}, there holds  
\begin{align} 
\frac{S_0(\ell)}{\ell} = \frac{4}{3} \EE \nu \norm{\grad u}_{L^2_x}^2 - \frac{4}{3} \eps + o_{\ell \rightarrow 0}(1),
\end{align} 
and the desired result for $S_0$ follows by the energy balance assumption \eqref{eq:condition1}. 

\subsection{Proof of (\ref{eq:scalinglaw}) for $S_{||}$}
By estimate~\eqref{eq:scalinglaw} for $S_0$,  
\begin{equation}
 2\ell^{-5} \int_0^\ell \tau^3 S_0(\tau)\dee\tau=o_{\ell \to 0}(1)
\end{equation}
where $o_{\ell \to 0}(1)$ vanishes as $\ell \to 0$ \emph{uniformly} with respect to $\nu$. Hence, from \eqref{eq:sperpoverr} and \eqref{eq:lim453}
we have 
\begin{equation}\label{eq:sperp-rel-to-H}
 \frac{S_{||}(\ell)}{\ell} = - 4\nu \frac{H(\ell)}{\ell} - \frac{4}{15}\eps + o_{\ell \to 0}(1).
\end{equation}
Hence, we consider
\begin{align} 
H(\ell)  & = \sum_{i,j,k} \frac{1}{4\pi} \EE \int_{\S^2} \int_{\T^3} \hat{n}^{i} \hat{n}^j \hat{n}^k   u^i(x) \partial_{x_k} u^j(x + \ell \hat{n}) \dd x \dd S(\hat{n})  \\ 
& = -\sum_{i,j,k} \frac{1}{4\pi} \EE \int_{\S^2} \int_{\T^3} \hat{n}^{i} \hat{n}^j \hat{n}^k  \partial_{x_k}  u^i(x) u^j(x + \ell \hat{n}) \dd x \dd S(\hat{n}).
\end{align} 
Observe that for all $i,j,k \in \{1,2,3\}$
\begin{align} 
\frac{1}{4\pi} \int_{\S^2}  \hat{n}^{i} \hat{n}^j \hat{n}^k  \dd S(\hat{n}) = 0, 
\end{align} 
and hence 
\begin{align} 
H(\ell) = - \sum_{i,j,k} \frac{1}{4\pi} \EE \int_{\S^2} \int_{\T^3} \hat{n}^{i} \hat{n}^j \hat{n}^k \partial_{x_k}  u^i(x) \left( u^j(x + \ell \hat{n}) - u^j(x) \right) \dd x \dd S(\hat{n}).
\end{align} 
By \eqref{eq:condition2} and \eqref{eq:EnergyIneq}, we have 
\begin{align} 
H(\ell) & = -\sum_{i,j,k,q} \frac{\ell}{4\pi} \EE \int_{\S^2} \int_{\T^3} \hat{n}^{i} \hat{n}^j \hat{n}^k  \hat{n}^{q} \, \partial_{x_k}  u^i(x)\partial_{x_q}u^j(x) \dx \dd S(\hat{n}) + \frac{\ell}{\nu}o_{\ell \to 0}(1),
\end{align} 
where $o_{\ell \to 0}(1)$ vanishes as $\ell \to 0$ \emph{uniformly} with respect to $\nu$.
 Then, we use the following identity for $i,j,k,q \in \set{1,2,3}$ and  the Kronecker $\delta$, 
 \begin{equation}
\frac{1}{4\pi}\int_{\S^2} \hat{n}^{i} \hat{n}^j \hat{n}^k\hat{n}^q   \dd S(\hat{n}) = \frac{1}{15}\left(\delta_{ij}\delta_{kq} + \delta_{ik}\delta_{jq} + \delta_{iq}\delta_{jk}\right), 
\end{equation}
so that
\begin{equation}
\sum_{i,j,k,q} \left(\frac{1}{4\pi}\int_{\S^2} \hat{n}^{i} \hat{n}^j \hat{n}^k\hat{n}^q   \dd S(\hat{n})\right)\left(\int_{\T^3} \partial_{x_k}  u^i(x)\partial_{x_q}u^j(x) \dx\right) = \frac{1}{15}\int_{\T^3} |\nabla u|^2\dx + \frac{2}{15} \int_{\T^3} (\Div u)^2 \dx.  
\end{equation}
Therefore, since $u$ is divergence free, 
\begin{equation} 
\nu\frac{H(\ell)}{\ell} = - \frac{\nu}{15}\EE \norm{\grad u}_{L^2_x}^2 + o_{\ell \to 0}(1).
\end{equation}
Hence, by \eqref{eq:sperp-rel-to-H} we obtain 
\begin{equation}
\frac{S_{||}(\ell)}{\ell} = \frac{4}{15}\nu \E\|\nabla u\|^2_{L^2_x} -\frac{4}{15}\eps + o_{\ell \to 0}(1), 
\end{equation}
Therefore, energy balance \eqref{eq:condition1} implies the result.

\section{Conclusions and further directions}
\label{sec:disc}
In the last section, we would like to suggest several potential directions in which to expand the study of \eqref{eq:NSE} in order to make closer contact with physicists and with experimental observations.
Even the basic assertion \eqref{eq:WAD} seems rather difficult to prove, but there are also a number of questions one can ask assuming \eqref{eq:WAD} and statistical symmetry assumptions. 
 Further, fairly reliable data is available on structure functions of order $p \leq 10$~\cite{vdWH99}; deducing further results assuming the experimental data on these quantities is approximately accurate may also be a reasonable direction. 
 We expect decaying turbulence problems, that is, deterministic Navier-Stokes with random initial data, to be significantly harder. Moreover, the initial data measures could add significant biases to the statistics of the answers if not chosen in a natural way (if there exists a natural way at all). 
 The following is a list of open important problems related \emph{specifically} to stationary martingale solutions to \eqref{eq:NSE}.

\subsection*{1. Anomalous dissipation}
 By far, the most important question is the existence of suitable stationary martingale solutions that satisfy \eqref{eq:WAD} (at least). It is important to further answer (affirmatively or negatively) whether or not such solutions are unique (in law). 
 
\subsection*{2. Quantitative regularity estimates} One can quantify estimates on the dissipation length scale and the inertial range regularity via the growth rate of various norms as $\nu \rightarrow 0$. 
In the physics literature,  quantities equivalent to the regularity $\EE \norm{u_{\gtrsim \ell_I^{-1}} }_{B_{p,\infty}^{\zeta_p/p}}^p$ for scaling exponents $\zeta_p < p$ and $p$ positive integers are traditionally considered\footnote{Here $u_{\gtrsim \ell_I}$ denotes the Littlewood-Paley projection to frequencies $\gtrsim \ell_I^{-1}$, $u_N$ denotes the Littlewood-Paley projection to frequencies $N/2 \lesssim \cdot \lesssim 2N$, and $B_{q,\infty}^s = \sup_{N \in 2^{\Natural}} \norm{\abs{\grad}^s u_N}_{L^q_x}$ denotes the inhomogeneous Besov space). It is classical that for $s \in (0,1)$ we have $\sup_{h \in B(0,1)} \abs{h}^{-s} \norm{\delta_{h} u}_{L^q_x}  + \norm{u}_{L^q_x} \approx \norm{u}_{B^{s}_{q,\infty}}$.}.
 The statistical self-similarity assumed in K41 formally predicts that such norms are uniformly bounded as $\nu \rightarrow 0$ for $\zeta_p = p/3$ for all $p \geq 2$, however, experimental observations suggest that $\zeta_2 > 2/3$ and $\zeta_p < p/3$ for $p > 3$. Due to the difficulty of measuring higher moments, reliable data seems only available for moments $p \leq 10$ (see e.g. \cite{AnselmetEtAl1984,Frisch1995,vdWH99} ). 
 Other reasonable norms could be, for example,  $\EE \norm{\e^{\lambda(\nu) \abs{\grad}} \brak{\grad}^\sigma u}_{L_x^p}^p$ (or the Gevrey analogues) where $\lim_{\nu \rightarrow 0}\lambda(\nu) = 0$ provides an estimate of the dissipation length-scale and $\sigma$ an estimate of the inertial range regularity. 
See for example the works of \cite{FlandoliRomito2002,Odasso06} which investigate what estimates follow from parabolic regularity and the energy inequality. 
 To our knowledge, no estimates have been made on $\set{u}_{\nu > 0}$ aside from those which follow from parabolic regularity combined with the energy inequality. 

\subsection*{3. Stochastic well-posedness at finite Reynolds numbers} Experimental measurements of the 6-th order structure function \cite{vdWH99} combined with e.g. \cite{FlandoliRomito2002} provides experimental evidence that statistically stationary solutions to \eqref{eq:NSE} could be almost-surely $C^{0,\alpha}_t C_x^\infty$ for all $\nu > 0$ (provided the force is smooth in space); note this also implies the energy balance, condition \eqref{eq:condition1}. 
In particular, this would suggest that the Navier-Stokes equations are indeed the correct equations with which to model turbulent fluids. 
Ergodicity, mixing, uniqueness and other similar properties of such solutions are also very important; see \cite{HM06,HM08,HM11} and the references therein for the corresponding work on 2D Navier-Stokes. 

\subsection*{4. Convex integration} It is extremely important to determine whether or not the $h$-principle holds in the same regularity classes that $\set{u}_{\nu > 0}$ could be uniformly bounded, e.g. in $L^3_t B_{3,\infty}^{1/3}$ (the regularity class matched by the 4/5 law). Furthermore, it is important to determine if `wild' solutions can be obtained via the inviscid limits of stationary solutions to \eqref{eq:NSE}.

\subsection*{5. Statistical symmetry} It is important to understand under what conditions one can prove statements such as the approximate isotropy in Definition \ref{def:ApproxSym}. Experimental evidence strongly suggests that the flows are not statistically scale invariant in the inertial range in any sense; however, one would want to determine under what conditions Parisi-Frisch multi-fractality \cite{FrischParisi1985,Frisch1995} and/or Kolmogorov's refined similarity hypotheses \cite{K62,StolovitzkyEtAl1992} can be verified in a mathematically rigorous framework. See \cite{CSint14} for work in this direction.  

\subsection*{6. Intermittency estimates} The flatness parameters 
\begin{align}
F_p(N) = \frac{\EE \norm{u_N}_{L^{2p}}^{2p}}{(\EE \norm{u_N}_{L^2}^2)^p}
\end{align} 
have been classically proposed as a measure of intermittency in both random and deterministic fields \cite{Frisch1995, CSint14}. Getting estimates on the growth of flatness parameters as $N \rightarrow \infty$ and $\nu \rightarrow 0$ will be important. 
Recall that statistically self-similar fields, Gaussian fields, and white noise all have bounded flatness parameters. 

\subsection*{7. Inviscid limit}  An important question is to determine whether or not we have the inviscid limit  $u^\nu \rightarrow u^0$ where $u^0$ is a statistically stationary weak solution of the stochastically forced Euler equations displaying a nonlinear energy flux balancing the input $\eps$ (see \cite{CV18,DV17} and the references therein for nearly sharp conditions in the deterministic case). 

\subsection*{8. Colored-in-time noise}
Eventually generalizing any results to cover at least colored-in-time random noise would be desirable (e.g. replacing $\dee \beta_k$ with an Ornstein-Uhlenbeck process). 
See e.g. \cite{KNS18} for results in this direction regarding ergodicity and mixing.

\end{document}